\documentclass[11pt,a4paper,article]{amsart}

\usepackage{amsmath,amsfonts,amssymb,amsthm}
\usepackage[english]{babel}
\usepackage[T1]{fontenc}
\usepackage[utf8]{inputenc}
\usepackage{xcolor}
\usepackage{datetime}
\usepackage{enumitem}
\usepackage[foot]{amsaddr}
\usepackage{lipsum}
\usepackage{pifont}
\usepackage{bbm}
\usepackage{ bbold }
\usepackage[dvipsnames]{xcolor}
\usepackage{comment}
\usepackage{tikz}
\usepackage{float}

\newtheorem{theorem}{Theorem}

\newtheorem*{lemma*}{Lemma}

\newtheorem{theo}{Theorem}[section]
\newtheorem{coro}[theo]{Corollary}
\newtheorem{propo}[theo]{Proposition}

\newtheorem{lem}[theo]{Lemma}

\theoremstyle{definition}
\newtheorem{remark}[theo]{Remark}

\def\finremark{	\hfill $\boxtimes$}

\numberwithin{equation}{section}

\definecolor{pRoot}{RGB}{224,130,20}   
\definecolor{pProt}{RGB}{33,102,172}   
\definecolor{pVert}{RGB}{27,158,119}   
\definecolor{cInt}{RGB}{43,47,54}      

\usepackage{xcolor}

\newcommand{\n}[1]{#1}

\newcommand{\convd}{\,{\buildrel \mathrm{d} \over \longrightarrow}\,}

\def\R{\mathbb{R}}

\def\Z{\mathbb{Z}}

\def\E{\mathbf{E}}
\def\V{\mathbf{V}}
\def\P{\mathbf{P}}

\def\T{\mathbb{T }}
\def\cT{\mathcal{T}}

    \title[Uniform integrability of the distance to the nearest leaf]{Uniform integrability of the distance to the nearest leaf in random trees}
\author[V.\,J. Maci\'a]{V\'{\i}ctor J. Maci\'a}
\address[V\'{\i}ctor J. Maci\'a]{Departamento de Análisis Matem\'atico, Universidad de La Laguna, Spain.}
\email{victor.macia@ull.edu.es}

\author[B. Stufler]{Benedikt Stufler}
\address[Benedikt Stufler]{TU Wien and Institute of Discrete Mathematics and Geometry}
\email{benedikt.stufler@tuwien.ac.at}
\subjclass{}
\keywords{Distance to the border, $k$-protected nodes, Kesten's tree, Galton-Watson processes, asymptotic distribution, convergence of moments.}

\begin{document}
	
\begin{abstract}
We study the distance from the root to the nearest leaf, the analogous quantity for a uniformly chosen vertex, and its protection number, in size-conditioned simply generated trees. We prove a uniform exponential tail bound for each of these quantities, valid for arbitrary offspring distributions. As a consequence, these random variables are uniformly integrable of every order. This yields convergence of all moments to those of the corresponding local limit. The argument is probabilistic and unified across the three quantities.
\end{abstract}
	\maketitle

	\section{Introduction}

	 For a rooted tree \(T\), the \emph{distance to the border}, \(\partial(T)\), is defined as the number of edges along the shortest path from the root to any leaf, while the \emph{height} of \(T\) is the number of edges along the longest such path. Here we use the convention that a leaf is defined as a vertex with no children. Together, these metric quantities capture the two extreme cases of root-to-leaf distances.
	\medskip
	
	Given a vertex $v$ of $T$, we likewise define $\partial(T,v)$ as the length of a shortest path from $v$ to any leaf of $T$. This notion is related to the protection number $\partial_{\mathrm{p}}(T,v)$, which is defined as the distance to the border in the fringe subtree of $T$ at $v$, that is, the subtree consisting of $v$ and all its descendants. The two notions are related through the inequality \[
	\partial(T,v) \le \partial_{\mathrm{p}}(T,v).\]
	 If $v$ is the root vertex of $T$, then the protection number coincides with the distance to the border.

		\medskip

	Devroye and Janson~\cite{zbMATH06346951} studied the protection number of a random node $v_n$ in a simply generated tree $\cT_n$ with $n$ vertices. That is,  given a sequence $(\omega_k)_{k \ge 0}$ of non-negative weights satisfying $\omega_0>0$ and $\omega_k>0$ for at least one $k \ge 2$, the random tree $\cT_n$ assumes a given $n$-vertex plane tree $T$ with vertex outdegrees $(d_{T}^+(v))_{v \in T}$ with probability
	\[
		\P({\cT_n = T}) \propto \prod_{v \in T} \omega_{d^+_T(v)}.
	\]
	We only consider $n$ where this model is well-defined, that is, for which the right-hand side is positive for at least one rooted $n$-vertex plane tree. Their main result~\cite[Thm. 1.1, Thm. 3.1]{zbMATH06346951} proves annealed and quenched distributional convergence of the protection number $\partial_{\mathrm{p}}(\cT_n, v_n)$ to the distance to the border $\partial(\cT)$ in an associated Bienaym\'e--Galton--Watson tree~$\cT$.
	
	\medskip
		
	A natural question is whether one also obtains convergence of moments. Using a generating function approach, this was verified by Cheon and Shapiro~\cite{Protection-Gisang} and Heuberger and Prodinger~\cite{Protection-Clemens} for the special case of uniform plane trees (i.e. $\omega_k = 1$ for all $k \ge 0$), and by Gittenberger, Gol\k{e}biewski, Larcher and Sulkowska~\cite{Protection1-mean} for simply generated trees whose weight-sequences are subject to an analytic subcriticality condition.
	
	\medskip
	
	Our main results describe in full generality the asymptotic limits of the protection number and distance to the border of typical vertices and the root vertex. That is, we treat all three cases I, II and III distinguished by Janson~\cite{Janson2012}. We verify uniform integrability through exponential bounds that ensure convergence of all moments. 
	\smallskip

	In the case of the distance to the border of a random vertex in the simply generated tree model we make use of a pointed random tree $\cT^{*}$ that describes the asymptotic vicinity of a random point up to its first ancestor with large degree. 

\begin{theorem}[Distance to the border of a random vertex]\label{te:border}
Let $v_n$ be a uniformly chosen vertex of $\cT_n$, and let $\cT^{*}$ be
the corresponding pointed local limit. As $n\to\infty$ along those $n$ for which
$\cT_n$ is defined,
\[
\partial(\cT_n,v_n) \xrightarrow{d} \partial(\cT^{*}),
\qquad
\E\bigl[\partial(\cT_n,v_n)^p\bigr] \to \E\bigl[\partial(\cT^{*})^p\bigr]
\quad(p\ge 1).
\]
\end{theorem}
	
	There is a critical regime for $\cT^{*}$ where we do not encounter such an ancestor, yielding an infinite spine, and a condensation regime where such an ancestor appears in the vicinity of a random node. In the proof we argue that the vicinity up to this ancestor is sufficient to determine the distance to the border. Since this vicinity converges in full generality by Stufler~\cite[Thm.~5.6]{zbMATH07039768}, the convergence in distribution in Theorem~\ref{te:border} holds for all three types with no additional assumption on the weight sequence, as do the exponential tail bound and the resulting uniform integrability.
	\smallskip
    
	We remark that although under some regularity assumptions a complete local limit up to the root is known~\cite{MR3826153}, it is an open research question to characterise what happens beyond the first ancestor with large degree in full generality because differing behaviour might occur. After a stochastically bounded number of generations we may encounter the root, or another ancestor with large degree, yielding possibilities for multiple or even infinitely many vertices with infinite degree to appear in the limit.
	
	For the classical question of what happens close to the root of $\cT_n$ this problem is better understood~\cite{MR0386042, MR607933, MR865130,MR871905,MR2764126,Janson2012} and a full classification is available, yielding that there is a random tree $\hat{\cT}$ that describes the asymptotic vicinity of the root vertex. 
	
\begin{theorem}[Distance to the border from the root]\label{te:root}
As $n\to\infty$ along those $n$ for which
$\cT_n$ is defined,
\[
\partial(\cT_n) \xrightarrow{d} \partial(\hat{\cT}),
\qquad
\E\bigl[\partial(\cT_n)^p\bigr] \to \E\bigl[\partial(\hat{\cT})^p\bigr]\quad(p\ge 1),
\]
where $\hat{\cT}$ is the local limit of $\cT_n$ around the root.
\end{theorem}
The tree $\hat{\mathcal{T}}$ is given by Kesten's tree for type~I and a tree with condensation for types~II and~III. For type~III the limit distribution for the distance to the border degenerates and $\partial(\hat{\cT})=1$ almost surely. Our proof strategy uses the convergence of $\cT_n$ towards $\hat{\cT}$ in the local sense to deduce distributional convergence of $\partial(\cT_n)$. 	The challenge lies in verifying the uniform integrability required to ensure convergence of all moments. To this end, we decompose along the path towards the left-most leaf, conditionally count and bound configurations and finally apply concentration inequalities for degree counts in simply generated trees.

\begin{theorem}[Protection number of a random vertex]\label{te:protection}
Let $v_n$ be a uniformly chosen vertex of $\cT_n$. As $n\to\infty$ along those $n$
for which $\cT_n$ is defined,
\[
\partial_{\mathrm{p}}(\cT_n,v_n) \xrightarrow{d} \partial(\cT),
\qquad
\E\bigl[\partial_{\mathrm{p}}(\cT_n,v_n)^p\bigr] \to \E\bigl[\partial(\cT)^p\bigr]
\quad(p\ge 1),
\]
where $\cT$ is a Bienaym\'e--Galton--Watson tree with a specific offspring distribution $\xi$.
\end{theorem}
	
	As mentioned above, the distributional convergence $\partial_{\mathrm{p}}(\cT_n,v_n) \xrightarrow{d} \partial(\cT)$ is due to Devroye and Janson~\cite{zbMATH06346951}. We contribute tail bounds that ensure that all moments converge as well.
	\smallskip
    
	The limit constants appearing in these results, in the critical case, are well known in the combinatorial literature, particularly in the study of \emph{$k$-protected nodes}; see, for instance, \cite{zbMATH06346951, Macia2022, Protection-Clemens, Protection1-mean, Protection-Gisang}. Their connection to the distance to the boundary of Kesten's tree, however, had not been previously identified. Rather than relying on traditional techniques in analytic combinatorics, such as the asymptotic analysis of generating functions, our approach leverages the properties of Kesten's tree and the condensation tree to derive these limit results in a direct and straightforward manner. We thereby provide a probabilistic interpretation of the underlying structure of the problem and an explicit description of the relevant quantities.

\smallskip

Our method generalises, simplifies, and unifies the derivation of these moment formulas, and bridges the probabilistic and combinatorial perspectives via the local limits. 

\smallskip

The analogous question for the \emph{height} of a tree is a classical topic that has been extensively investigated in both the probabilistic and combinatorial literature~\cite{Renyi,deBruijnKnuthRice,FlajoletOdlyzko,zbMATH07940224,MR3077536,MR1085326}.
Although the height and the distance to the border are extreme statistics of the same type, longest versus shortest root-to-leaf path, they exhibit very different asymptotic behaviour: the height typically grows  as the number of vertices tends to infinity, while the distance to the border remains tight. The present paper is concerned exclusively with the latter quantity. 

\subsection*{Organisation of the paper}

The paper is organised as follows. Section~\ref{sec:sgt} introduces the simply generated tree model, the associated probability measure, and the classification of weight sequences into types~I, II and~III, following Janson~\cite{Janson2012}. Section~\ref{sec:diboro} treats the distance to the border from the root vertex: we recall the construction of the relevant local limits, Kesten's tree and the condensation tree, prove convergence in distribution of $\partial(\cT_n)$ to the corresponding limit in each regime, and establish a uniform exponential tail bound for $\partial(\cT_n)$, from which uniform integrability and convergence of all moments follow. Sections~\ref{sec:protection_random_vertex} and~\ref{sec:distance_random_vertex} carry out the parallel programme for the protection number $\partial_{\mathrm{p}}(\cT_n,v_n)$ and the distance to the border $\partial(\cT_n,v_n)$ of a uniformly chosen vertex; the additional ingredient is a chain-of-first-descendants construction accounting for the random choice of vertex. The paper closes with a section collecting a few applications.

\subsection*{Acknowledgements}
The first author gratefully acknowledges Prof. José L. Fernández for bringing this problem to his attention and Prof. Christina Goldschmidt for helpful discussions. He also thanks the Institute for Discrete Mathematics and Geometry at TU Wien for its hospitality and support.

\subsection*{Funding} This research was funded in part by the Austrian Science Fund (FWF) 10.55776/F1002 and 10.55776/PAT6732623. For open access purposes, the authors have applied a CC BY public copyright license to any author-accepted manuscript version arising from this submission. Research of V.~J.\ Maci\'a was partially funded by grants PID2021-123151NB-I00, PID2023-148028NB-I00 and PID2025-167726NB-I00 from the Spanish Government.

\section{Notations}
Let $(X_n)_{n\ge1}$ be a sequence of random variables and $X$ a random variable. We write
\[
X_n \xrightarrow{d} X \qquad \text{as } n\to\infty
\]
for convergence in distribution of $X_n$ to $X$. On a probability space
$(\Omega,\mathcal F,\mathbb P)$ we denote by $\E$ and $\V$ the expectation and variance, respectively. Unspecified limits are always taken as $n \to \infty$.

\vspace{.2 cm}

\noindent\emph{Boundary functionals.}
For a rooted plane tree $T$ we write $\mathrm{leaves}(T)$ for its set of leaves and
$d(\cdot,\cdot)$ for the graph distance. The  \emph{distance to the
border} of $T$ is given by
\[
\partial(T) = \min_{u \in \mathrm{leaves}(T)} d(\mathrm{root}, u).
\]
For infinite trees $T$ it can happen that $\mathrm{leaves}(T)=\emptyset$. In this case we set $\partial(T) = \infty$. More generally, for a vertex $v$ of $T$,
\[
\partial(T,v) = \min_{u \in \mathrm{leaves}(T)} d(v, u),
\]
is the \emph{distance to the border from $v$}, so that $\partial(T)=\partial(T,\mathrm{root})$.
Writing $\mathrm{fringe}_v(T)$ for the fringe subtree of $T$ at $v$ (the subtree consisting of
$v$ and all its descendants), the \emph{protection number} of $v$ is the distance to the
border within that fringe,
\[
\partial_{\mathrm{p}}(T,v) = \partial\bigl(\mathrm{fringe}_v(T)\bigr)
= \min_{u \in \mathrm{leaves}(\mathrm{fringe}_v(T))} d(v,u).
\]
Every leaf of $\mathrm{fringe}_v(T)$ is a leaf of $T$, while $v$ may also reach leaves of $T$
outside its fringe; hence $\partial(T,v) \le \partial_{\mathrm{p}}(T,v)$
(see Figure~\ref{fig:three_statistics}).

\bigskip
\noindent\emph{Random trees and their local limits.}
In types~I and~II, let $\cT$ be the Bienaym\'e--Galton--Watson tree with
offspring distribution $\xi$ associated with the weight sequence
$(\omega_k)_{k \ge 0}$ through the tilting of Section~\ref{sec:sgt}.
In type~III, let $\cT$ consist of a single vertex, consistently with our
convention that $\xi=0$ almost surely. Let $\cT_n$ be the simply generated
tree of size $n$.
\vspace{.2 cm}
\[
\left\{
\begin{array}{lcl}
	\displaystyle \cT &\colon& \text{the Bienaym\'e--Galton--Watson tree with offspring distribution }\xi,\\[1mm]
	\displaystyle \hat{\cT} &\colon& \text{the local limit of $\cT_n$ around the root},\\[1mm]
	\displaystyle \cT^{*} &\colon& \text{the local limit of $\cT_n$ around a uniformly chosen vertex (pointed limit).}
\end{array}
\right.
\]
\vspace{.2 cm}

\noindent See Section~\ref{sec:local_limits} for the construction of $\hat{\mathcal{T}}$ and Section~\ref{sec:tstar} for the definition of $\mathcal{T}^*$.

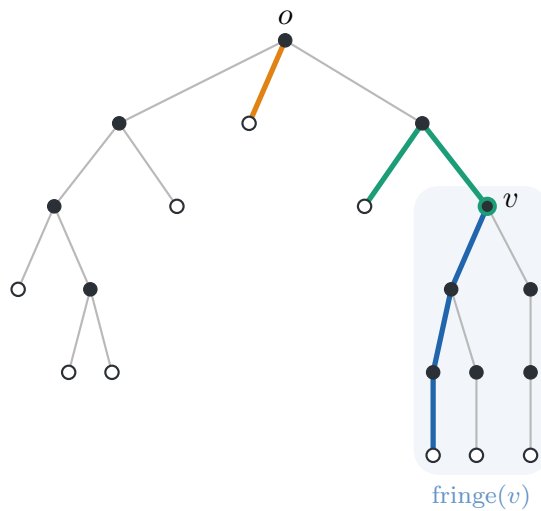
\begin{figure}[H]
\centering
\resizebox{0.5\linewidth}{!}{%
\begin{tikzpicture}[scale=0.82,line cap=round,line join=round,
  int/.style ={circle,draw=cInt,fill=cInt,minimum size=4.3pt,inner sep=0},
  leaf/.style={circle,draw=cInt,fill=white,line width=.7pt,minimum size=4.3pt,inner sep=0},
  ed/.style  ={gray!55,line width=0.7pt},
  hl/.style  ={line width=1.7pt}]
\fill[pProt!6,rounded corners=8pt] (1.78,-2.02) rectangle (3.66,-6.02);
\node[pProt!70,font=\scriptsize] at (2.72,-6.32) {fringe$(v)$};
\coordinate (o)  at (0,0);
\coordinate (L)  at (-2.3,-1.15); \coordinate (m) at (-0.5,-1.15); \coordinate (p) at (1.9,-1.15);
\coordinate (L1) at (-3.2,-2.3);  \coordinate (L2) at (-1.5,-2.3);
\coordinate (b)  at (1.1,-2.3);   \coordinate (v) at (2.8,-2.3);
\coordinate (L11) at (-3.7,-3.45);\coordinate (L12) at (-2.7,-3.45);
\coordinate (v1) at (2.3,-3.45);  \coordinate (v2) at (3.4,-3.45);
\coordinate (L121) at (-3.0,-4.6);\coordinate (L122) at (-2.4,-4.6);
\coordinate (v11) at (2.05,-4.6); \coordinate (v12) at (2.65,-4.6); \coordinate (v21) at (3.4,-4.6);
\coordinate (d1) at (2.05,-5.75); \coordinate (d2) at (2.65,-5.75); \coordinate (d3) at (3.4,-5.75);
\draw[ed] (o)--(L); \draw[ed] (o)--(m); \draw[ed] (o)--(p);
\draw[ed] (L)--(L1); \draw[ed] (L)--(L2);
\draw[ed] (L1)--(L11); \draw[ed] (L1)--(L12); \draw[ed] (L12)--(L121); \draw[ed] (L12)--(L122);
\draw[ed] (p)--(b); \draw[ed] (p)--(v);
\draw[ed] (v)--(v1); \draw[ed] (v)--(v2);
\draw[ed] (v1)--(v11); \draw[ed] (v1)--(v12); \draw[ed] (v2)--(v21);
\draw[ed] (v11)--(d1); \draw[ed] (v12)--(d2); \draw[ed] (v21)--(d3);
\draw[pProt,hl] (v)--(v1)--(v11)--(d1);
\draw[pVert,hl] (v)--(p)--(b);
\draw[pRoot,hl] (o)--(m);
\foreach \n in {o,L,p,L1,L12,v1,v2,v11,v12,v21}{ \node[int] at (\n){}; }
\foreach \n in {m,L2,b,L11,L121,L122,d1,d2,d3}{ \node[leaf] at (\n){}; }
\draw[pVert,line width=1.3pt,fill=cInt] (v) circle (3.0pt);
\node[font=\small] at (0,0.32) {$o$};
\node[font=\small] at (3.02,-2.22) {\, $v$};
\node[anchor=west,font=\footnotesize] at (-3.7,-6.92)
  {\textcolor{pRoot}{$\partial(T)=1$}\quad\textcolor{pVert}{$\partial(T,v)=2$}\quad\textcolor{pProt}{$\partial_{\mathrm p}(T,v)=3$}};
\end{tikzpicture}%
}
\caption{The three boundary statistics on a plane tree (internal vertices filled, leaves
white; the marked vertex $v$ is ringed and its fringe subtree shaded): $\partial(T)$,
$\partial(T,v)$ and $\partial_{\mathrm p}(T,v)$. }
\label{fig:three_statistics}
\end{figure}

\section{Simply generated trees}\label{sec:sgt}
In this section we introduce the simply generated tree model,
recall the classification of weight sequences into
types~I, II and~III following Janson~\cite{Janson2012} and clarify the connection to
Bienaym\'e--Galton--Watson trees.

\subsection{The model}

Let $\mathbf{w} = (\omega_k)_{k \ge 0}$ be a sequence of non-negative weights
with $\omega_0 > 0$ and $\omega_k > 0$ for at least one $k \ge 2$.
Let
\[
\Phi(z) := \sum_{k=0}^{\infty} \omega_k z^k
\]
be a power series, and let $R \in [0,\infty]$ denote its radius of
convergence. For each $n \ge 1$ such that at least one rooted plane tree
with $n$ vertices has positive weight, the \emph{simply generated tree}
$\cT_n$ is the random rooted plane tree with $n$ vertices whose law
assigns to every $n$-vertex rooted plane tree $T$ a probability
proportional to its \emph{weight}
\[
w(T) := \prod_{v \in T} \omega_{d^+_T(v)},
\]
where $d^+_T(v)$ denotes the out-degree of vertex $v$ in $T$:
\begin{equation}\label{eq:sgt_prob}
	\P({\cT_n = T})
	= \frac{w(T)}{\displaystyle\sum_{\substack{T':\, |T'|=n}} w(T')}.
\end{equation}
Throughout this paper we will implicitly restrict $n$ to values for which~\eqref{eq:sgt_prob} is
well-defined, i.e., for which the denominator is strictly positive.

\subsection{Relation to Bienaym\'e--Galton--Watson trees}

We shall use the following terminology. Let $\mathcal K$ denote the
class of non-constant power series
\[
f(z)=\sum_{k=0}^{\infty} a_k z^k
\]
with positive radius of convergence, non-negative Taylor coefficients,
and $a_0>0$. If $f\in\mathcal K$ has radius of convergence $R_f>0$,
its \emph{Khinchin family} is the family of random variables
$(X_t)_{t\in[0,R_f)}$ with values in $\mathbb N_0 = \{0,1,2,\dots\}$ and distribution
\[
\P(X_t=k)=\frac{a_k t^k}{f(t)}, \qquad k\geq 0.
\]

In the present setting, whenever $R>0$, the weight generating function
$\Phi$ belongs to $\mathcal K$, and we denote by
$(X_t)_{t\in(0,R)}$ its associated Khinchin family:
\[
\P(X_t = k) \;=\; \frac{\omega_k\,t^k}{\Phi(t)}, \qquad k\geq 0.
\]
The mean offspring is
\[
m_\Phi(t) \;:=\; \E[X_t]
\;=\; \frac{t\,\Phi'(t)}{\Phi(t)}.
\]

A fundamental feature of the model is that, for every $t\in(0,R)$ (or $t=R$ if $\Phi(R)<\infty$),
the law of $\cT_n$ coincides with the law of a $X_t$-BGW tree $\cT$
conditioned on having $n$ vertices.

\subsection{Types of weight sequences}\label{subsec:types}

We follow the classification of~\cite{Janson2012} and~\cite{zbMATH07039768}. In all cases we define an offspring distribution $\xi$ for a BGW $\mathcal{T}$.
\medskip

\paragraph{Type~I}
The weight sequence $\mathbf{w}$ is of \emph{type~I} if there exists
$\tau \in (0,R]$ such that $\Phi(\tau)<\infty$ and
$m_\Phi(\tau)=1$. In this case $X_\tau$ is critical, and we set
$\xi:=X_\tau$. 

\smallskip

\paragraph{Type~II}
The weight sequence $\mathbf{w}$ is of \emph{type~II} if $R>0$,
$\Phi(R)<\infty$, and
\[
m_\Phi(R) \;=\; \frac{R\,\Phi'(R)}{\Phi(R)} \;<\; 1.
\]
In this case the Khinchin family extends to the boundary point $t=R$,
where
\[
\P(X_R = k) \;=\; \frac{\omega_k R^k}{\Phi(R)}, \qquad k\geq 0,
\]
defines a subcritical offspring distribution. We take $\xi:=X_R$.

\smallskip

\paragraph{Type~III}
The weight sequence $\mathbf{w}$ is of \emph{type~III} if $R=0$. In this case
$\Phi$ does not belong to $\mathcal K$, and no Khinchin family is
available in the above sense. It will be notationally convenient to set $\xi=0$, so that $\mathcal{T}$ consists of a single vertex.

\section{The distance to the border from the root vertex}

\label{sec:diboro}

In this section we study the distance to the border $\partial(\cT_n)$ of a
size-conditioned simply generated tree, measured from its root, in the three regimes of
Section~\ref{subsec:types}. After fixing the topological setting on the spaces of plane
trees $(\T,\delta)$ and $(\T_{\infty},\delta_{\infty})$, we prove that $\partial(\cT_n)$
converges in distribution to $\partial(\hat{\cT})$, the distance to the border of the
local limit $\hat{\cT}$ around the root, by combining the local convergence
$\cT_n\xrightarrow{d}\hat{\cT}$ with the almost-sure continuity of the functional
$\partial$ under the limiting law. We then establish a uniform exponential tail bound for
$\partial(\cT_n)$, valid in all three types and with no assumption on the weight sequence,
from which uniform integrability and the convergence of all moments of $\partial(\cT_n)$
to those of $\partial(\hat{\cT})$ follow.
\subsection{Topology}
We denote by \(\T\) the set of rooted plane trees in which every node has
a finite number of children (although the trees themselves may be infinite),
and by \(\T_0\) the set of finite rooted plane trees (i.e., trees with a
finite number of nodes). We use the Ulam--Harris labelling and Neveu's
formalism; see, for instance, \cite{Neveu,AbrahamDelmasnotes}
for further details. It is well known that \(\T\), endowed with the metric
\[
\delta(T,T') = 2^{-\sup\{h \geq 0 \, : \, r_h(T) = r_h(T')\}},
\]
is a Polish space (see, for instance, \cite{AbrahamDelmasnotes}). Here, \(r_h(T)\) denotes the restriction function that gives the tree obtained by pruning \(T\) at level \(h\) (i.e., keeping the nodes at level \(h\) as leaves). 
\smallskip

We denote $\T_{\infty}$ the set of rooted plane trees (we allow nodes with an infinite number of children), see \cite{AbrahamDelmasnotes} and \cite{Neveu}. Using Neveu's formalism, see \cite{Neveu}, for a node $u = (u_1,\dots,u_n) \in (\mathbb{N} \setminus\{0\})^{n}$ denote $|u| = n$ its length and define
\begin{align*}
	|u|_{\infty} = \max\{|u|,u_1,\dots,u_{|u|}\}\,,
\end{align*}
with the convention $|\emptyset| = |\emptyset|_{\infty} = 0$. Using $|u|_{\infty}$ we define the restriction map ${r}_{h,\infty}:\T_{\infty} \rightarrow \T_{\infty}$ as 
\begin{align*}
	{r}_{h,\infty}(T) = \{u \in T : |u|_{\infty} \leq h\}\,,
\end{align*}
that is, we keep only the nodes at depth at most $h$ whose Ulam--Harris coordinates are all bounded by $h$. Define 
\begin{align*}
	\delta_{\infty}(T,T') = 2^{-\sup\{h \geq 0 \, : \, r_{h,\infty}(T) = r_{h,\infty}(T')\}}\,.
\end{align*}
Equipped with this metric, $(\T_{\infty},\delta_{\infty})$ is a compact Polish metric space; see, for instance, \cite{AbrahamDelmas} for further details.	
\smallskip

Following \cite{AbrahamDelmas} and \cite{AbrahamDelmasnotes} we denote
\[
\T_0^{\star}
:=
\bigl\{
T \in \T_{\infty}
:\ \text{$T$ has no infinite branch}
\bigr\}.
\]
and
\[
\T_2
:=
\Bigl\{
T \in \T_{\infty}
:\ \#\{u \in T : k_u(T)=+\infty\}=1
\Bigr\}
\cap
\T_0^{\star}\,.
\]
\smallskip

For a rooted tree \(T\), we denote by \(\mathrm{leaves}(T)\) its set of leaves, by \(|T|\) the number of nodes in \(T\), and by \(k_u(T)\) the number of children of the node \(u\) in \(T\) (with the convention that if \(u \notin T\), then \(k_u(T) = -1\)). Throughout, we let \(d(u,v)\) denote the standard graph distance between vertices \(u\) and \(v\) (i.e., the number of edges in the shortest path connecting them).
\smallskip

\subsection{The limiting objects}

\label{sec:local_limits}

We recall Janson's unified construction of the local limit of $\cT_n$ around the
root, valid in all three regimes; see \cite[\S5 and Thm.~7.1]{Janson2012}. In types~I and~II, let $(p_k)_{k\ge0}$ be the canonical probability weight sequence
equivalent to the weight sequence (Section~\ref{subsec:types}), with offspring variable
$\xi$, probability distribution $\P(\xi=k)=p_k$, and mean $m:=\E(\xi)\le1$.
In type~III, consistently with our convention above, we set $p_0=1$ and $\xi=0$
almost surely, so that $m=0$. Write $\cT$ for the Bienaym\'e--Galton--Watson tree with offspring
$\xi$.
\smallskip

Following Kesten and Jonsson--Stef\'ansson (see \cite[\S5]{Janson2012}), define a
random infinite tree $\hat{\cT}$ with two kinds of vertices, \emph{normal} and
\emph{special}, the root being special. Normal vertices have offspring given by
independent copies of $\xi$, while special vertices have offspring given by
independent copies of $\hat{\xi}$, where
\[
\P(\hat{\xi}=k)=
\begin{cases}
	k\,p_k, & k=0,1,2,\dots,\\[1mm]
	1-m, & k=\infty,
\end{cases}
\]
a probability distribution on $\{1,2,\dots\}\cup\{\infty\}$. All children of a
normal vertex are normal; if a special vertex has infinitely many children, all are
normal; if it has $k<\infty$ children, exactly one of them, chosen uniformly at
random, is special and the remaining $k-1$ are normal. The special vertices form a
path from the root, the \emph{spine} of $\hat{\cT}$. By \cite[Thm.~7.1]{Janson2012},
$\cT_n\xrightarrow{d}\hat{\cT}$ in the local topology, along those $n$ for which
$\cT_n$ is defined.
\smallskip

The classification into types records the qualitative behaviour of $\hat{\cT}$:
\begin{itemize}\setlength{\itemsep}{0.7em}
	\item In \emph{type~I} ($m=1$) we have $\hat{\xi}<\infty$ almost surely, so every special
	individual has a special child: the spine is an infinite path and $\hat{\cT}$ is locally
	finite. In this case $\hat{\xi}$ is the size-biased law of $\xi$ and $\hat{\cT}$ is
	Kesten's size-biased Galton--Watson tree.
	\item In \emph{type~II} ($0<m<1$) a special individual has no special child with
	probability $1-m$, so the spine is almost surely finite, its length following the
	geometric law $\P(G=\ell)=(1-m)\,m^{\ell-1}$, $\ell\ge1$. The top of the spine is almost
	surely the unique vertex of infinite outdegree, and $\hat{\cT}$ is the condensation tree
	of Jonsson--Stef\'ansson; see also \cite{AbrahamDelmasnotes,AbrahamDelmas}.
	\item In \emph{type~III} ($m=0$) we have $\hat{\xi}=\infty$ almost surely, so the root is
	the only special individual and $\hat{\cT}$ is the infinite star: the root carries
	infinitely many children, all of them leaves. In particular $\partial(\hat{\cT})=1$
	almost surely.
\end{itemize}

\begin{figure}[H]
	\centering
	\resizebox{0.9\linewidth}{!}{%
		\begin{tikzpicture}[line cap=round,line join=round,
			int/.style ={circle,draw=cInt,fill=cInt,minimum size=4.5pt,inner sep=0},
			leaf/.style={circle,draw=cInt,fill=white,line width=.7pt,minimum size=4.5pt,inner sep=0},
			ed/.style  ={gray!55,line width=0.7pt},
			sp/.style  ={line width=1.05pt,gray!70},
			gw/.style  ={fill=gray!16,draw=gray!50,line width=.5pt},
			near/.style={pRoot,line width=1.6pt}]
			\begin{scope}[shift={(0,0)}]
				\node[font=\bfseries] at (0,1.0) {type I};
				\coordinate (o) at (0,0); \coordinate (u1) at (0,-1.65); \coordinate (u2) at (0,-3.3);
				\draw[sp] (o)--(u1)--(u2); \draw[sp,dashed] (u2)--(0,-4.1); \node at (0,-4.4) {$\vdots$};
				\draw[ed] (o)--(0.72,-0.8);  \fill[gw] (0.72,-0.8)--(0.46,-1.35)--(0.98,-1.35)--cycle;   \node[font=\scriptsize,gray!55!black] at (0.72,-1.1){$\cT$};
				\draw[ed] (o)--(1.78,-0.8);  \fill[gw] (1.78,-0.8)--(1.52,-1.35)--(2.04,-1.35)--cycle;   \node[font=\scriptsize,gray!55!black] at (1.78,-1.1){$\cT$};
				\node[font=\footnotesize] at (1.25,-0.86) {$\cdots$};
				\draw[ed] (o)--(-1.78,-0.8); \fill[gw] (-1.78,-0.8)--(-2.04,-1.35)--(-1.52,-1.35)--cycle; \node[font=\scriptsize,gray!55!black] at (-1.78,-1.1){$\cT$};
				\node[font=\footnotesize] at (-1.25,-0.86) {$\cdots$};
				\draw[ed] (o)--(-0.72,-0.8); \fill[gw] (-0.72,-0.8)--(-0.98,-1.35)--(-0.46,-1.35)--cycle; \node[font=\scriptsize,gray!55!black] at (-0.55,-1.13){$\cT$};
				\draw[ed] (u1)--(0.72,-2.45);  \fill[gw] (0.72,-2.45)--(0.46,-3.0)--(0.98,-3.0)--cycle;   \node[font=\scriptsize,gray!55!black] at (0.72,-2.75){$\cT$};
				\draw[ed] (u1)--(1.78,-2.45);  \fill[gw] (1.78,-2.45)--(1.52,-3.0)--(2.04,-3.0)--cycle;   \node[font=\scriptsize,gray!55!black] at (1.78,-2.75){$\cT$};
				\node[font=\footnotesize] at (1.25,-2.51) {$\cdots$};
				\draw[ed] (u1)--(-0.72,-2.45); \fill[gw] (-0.72,-2.45)--(-0.98,-3.0)--(-0.46,-3.0)--cycle; \node[font=\scriptsize,gray!55!black] at (-0.72,-2.75){$\cT$};
				\draw[ed] (u1)--(-1.78,-2.45); \fill[gw] (-1.78,-2.45)--(-2.04,-3.0)--(-1.52,-3.0)--cycle; \node[font=\scriptsize,gray!55!black] at (-1.78,-2.75){$\cT$};
				\node[font=\footnotesize] at (-1.25,-2.51) {$\cdots$};
				\coordinate (w) at (-0.72,-0.8); \coordinate (l) at (-0.8,-1.2);
				\draw[near] (o)--(w)--(l);
				\node[int] at (o){}; \node[int] at (u1){}; \node[int] at (u2){}; \node[int] at (w){}; \node[leaf] at (l){};
				\node[font=\small,anchor=south] at (0,0.12) {$o$};
				\node[pRoot,font=\small,anchor=east] at (-1.45,-1.55) {$\partial(\hat{\cT})$};
			\end{scope}
			\begin{scope}[shift={(5.6,0)}]
				\node[font=\bfseries] at (0,1.0) {type II};
				\coordinate (o) at (0,0); \coordinate (s1) at (0,-1.5); \coordinate (s2) at (0,-3.6);
				\draw[sp] (o)--(s1)--(0,-2.7); \node at (0,-2.95) {$\vdots$}; \draw[sp] (0,-3.2)--(s2);
				\draw[ed] (o)--(0.65,-0.55);  \fill[gw] (0.65,-0.55)--(0.41,-1.12)--(0.89,-1.12)--cycle;   \node[font=\scriptsize,gray!55!black] at (0.65,-0.86){$\cT$};
				\draw[ed] (o)--(1.7,-0.55);   \fill[gw] (1.7,-0.55)--(1.46,-1.12)--(1.94,-1.12)--cycle;     \node[font=\scriptsize,gray!55!black] at (1.7,-0.86){$\cT$};
				\node[font=\footnotesize] at (1.16,-0.6) {$\cdots$};
				\draw[ed] (o)--(-0.65,-0.55); \fill[gw] (-0.65,-0.55)--(-0.89,-1.12)--(-0.41,-1.12)--cycle; \node[font=\scriptsize,gray!55!black] at (-0.65,-0.86){$\cT$};
				\draw[ed] (o)--(-1.7,-0.55);  \fill[gw] (-1.7,-0.55)--(-1.94,-1.12)--(-1.46,-1.12)--cycle;  \node[font=\scriptsize,gray!55!black] at (-1.7,-0.86){$\cT$};
				\node[font=\footnotesize] at (-1.16,-0.6) {$\cdots$};
				\draw[ed] (s1)--(0.65,-2.05);  \fill[gw] (0.65,-2.05)--(0.41,-2.62)--(0.89,-2.62)--cycle;   \node[font=\scriptsize,gray!55!black] at (0.65,-2.36){$\cT$};
				\draw[ed] (s1)--(1.7,-2.05);   \fill[gw] (1.7,-2.05)--(1.46,-2.62)--(1.94,-2.62)--cycle;     \node[font=\scriptsize,gray!55!black] at (1.7,-2.36){$\cT$};
				\node[font=\footnotesize] at (1.16,-2.1) {$\cdots$};
				\draw[ed] (s1)--(-0.65,-2.05); \fill[gw] (-0.65,-2.05)--(-0.89,-2.62)--(-0.41,-2.62)--cycle; \node[font=\scriptsize,gray!55!black] at (-0.65,-2.36){$\cT$};
				\draw[ed] (s1)--(-1.7,-2.05);  \fill[gw] (-1.7,-2.05)--(-1.94,-2.62)--(-1.46,-2.62)--cycle;  \node[font=\scriptsize,gray!55!black] at (-1.7,-2.36){$\cT$};
				\node[font=\footnotesize] at (-1.16,-2.1) {$\cdots$};
				\foreach \x in {-1.5,-1.0,-0.5,0.5,1.0,1.5}{ \draw[ed] (s2)--(\x,-4.8); }
				\node at (0,-5.0) {$\cdots$};
				\draw[near] (o)--(s1)--(0,-2.7); \draw[near] (0,-3.2)--(s2)--(-0.5,-4.8);
				\node[int] at (o){}; \node[int] at (s1){}; \node[int] at (s2){};
				\foreach \x in {-1.5,-1.0,-0.5,0.5,1.0,1.5}{ \node[leaf] at (\x,-4.8){}; }
				\node[font=\small,anchor=south] at (0,0.12) {$o$};
				\node[font=\footnotesize,anchor=west] at (0.25,-3.7) {$\deg=\infty$};
				\node[pRoot,font=\small,anchor=east] at (-1.15,-4.3) {$\partial(\hat{\cT})$};
			\end{scope}
			\begin{scope}[shift={(10.8,0)}]
				\node[font=\bfseries] at (0,1.0) {type III};
				\coordinate (o) at (0,0);
				\foreach \x in {-1.5,-1.0,-0.5,0.5,1.0,1.5}{ \draw[ed] (o)--(\x,-1.3); }
				\node at (0,-1.5) {$\cdots$};
				\draw[near] (o)--(-0.5,-1.3);
				\node[int] at (o){};
				\foreach \x in {-1.5,-1.0,-0.5,0.5,1.0,1.5}{ \node[leaf] at (\x,-1.3){}; }
				\node[font=\small,anchor=west] at (0.12,0.04) {$o$};
				\node[font=\footnotesize,anchor=west] at (1.05,-0.5) {$\deg=\infty$};
				\node[pRoot,font=\small,anchor=east] at (-0.72,-0.5) {$\partial(\hat{\cT})=1$};
			\end{scope}
		\end{tikzpicture}%
	}
	\caption{The root local limit $\hat{\cT}$ in the three regimes, with the distance to the
		border $\partial(\hat{\cT})$ highlighted.}
	\label{fig:local_limits}
\end{figure}
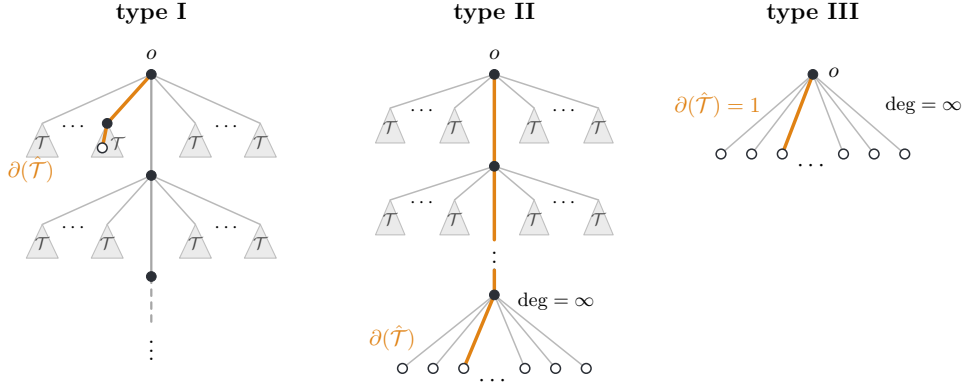

\subsection{Convergence in distribution}\label{sec:conv_dist_root}
We prove that $\partial(\cT_n)$ converges in distribution to $\partial(\hat{\cT})$ for weight
sequences of all three types, where $\hat{\cT}$ is the local limit around the root described
in Section~\ref{sec:local_limits}. The argument is uniform: by Janson~\cite[Thm.~7.1]{Janson2012}
one has $\cT_n\xrightarrow{d}\hat{\cT}$ in the local topology, and the conclusion follows from
the mapping theorem once we know that the functional $\partial$ is almost surely continuous
under the limiting law. The only step that depends on the type is the verification that the
discontinuity set of $\partial$ is negligible, which we carry out case by case. We recall that
throughout this paper every offspring distribution satisfies $p_0+p_1<1$ (except in the type III case where $p_0 = 1$).
\smallskip

In the locally finite setting $(\T,\delta)$ the event $\{\partial\ge k\}$ is in fact clopen,
which gives at once the Borel measurability of $\partial$ and, in the critical case, a direct
proof; we record this for later use.

\begin{lem}\label{lem: clopen} Fix an integer $k \geq 0$. The set of trees $\T_k:=\{T\in\T:\partial(T)\ge k\}$ is an open and closed set of the metric space $(\T,\delta)$. Consequently the map $\partial : \T \rightarrow \mathbb{N} \cup \{0\}$ is Borel measurable. 
\end{lem}
\begin{proof}
	Fix an integer $k\ge0$. To prove that $\T_k:=\{T\in\T:\partial(T)\ge k\}$ is open in $(\T,\delta)$, let $T\in\T_k$ and define $U:=\{T'\in\T:r_k(T)=r_k(T')\}$. By definition, $U$ is an open neighborhood of $T$ (see, e.g., \cite[p.~6]{AbrahamDelmasnotes}); indeed, $U=B_\delta(T,2^{-(k-1)})$. Since every tree in $U$ coincides with $T$ at least up to level $k$, it follows that $U\subseteq\T_k$.
	\smallskip
	
	Next we prove that $\T_k$ is a closed set of $(\T,\delta)$. Let $(T_n)_{n \geq 1} \subset \T_k$ be a sequence such that $T_n \rightarrow T$, as $n \rightarrow \infty$. It can be proved, see, for instance, \cite[p.~6]{AbrahamDelmasnotes}, that $T_n$ converges to $T$, as $n \rightarrow \infty$, if and only if $k_u(T_n) \rightarrow k_u(T)$ for every node $u$, as $n \rightarrow \infty$.
	\smallskip
	
	Let $u$ be any node of $T$ below generation $k$ and assume that $k_u(T) = 0$. The convergence of the sequence $T_n$ implies that there exists an integer $N>0$ such that $|k_u(T_n)|<1$ for every $n \geq N$. This in turn implies that $k_u(T_n) = 0$ for any $n \geq N$, so that $T_n \notin \T_k$ for any $n \geq N$, which is a contradiction. Therefore $T \in \T_k$ and $\T_k$ is closed.
\end{proof}

We endow $\mathbb{R} \cup \{\infty\}$ with the usual one-point compactification topology.
\begin{propo}\label{Propo: discontinuity_set_delta_condensation} The set of discontinuities of $\partial : \T_{\infty} \rightarrow \R \cup \{\infty\}$ is given by 
\[
\mathrm{Disc}(\partial)
=
\Bigl\{\,T \in \T_{\infty}:\ \exists v \ \text{such that } k_v(T)=\infty \ \text{and}\ \partial(T)>|v|+1 \Bigr\}.
\]
\end{propo}
\begin{proof}
Let $T \in \T_{\infty}$ and assume that there exists a vertex $v$ with
$k_v(T)=\infty$ and $\partial(T)>|v|+1$ (equivalently, there are no leaves at height
$\le |v|+1$ in $T$). We construct a sequence $(T_n)_{n\ge1}$ with $T_n\to T$ but
$\partial(T_n)=|v|+1$ for all $n$, hence $\partial$ is discontinuous at $T$.
\smallskip

Enumerate the children of $v$ in $T$ as $(v1,v2,\dots)$ according to the plane
ordering. For each $n\ge1$, define $T_n$ as follows:
\begin{enumerate}
\item the tree $T_n$ coincides with $T$ everywhere except possibly at the vertex $v$;
\item the vertex $v$ keeps its first $n-1$ children as in $T$, together with
their full descendant subtrees;
\item the $n$-th child of $v$ in $T_n$ is declared to be a leaf (i.e.\ it has out-degree $0$);
\item the remaining children of $v$ (with indices $>n$) are chosen so that the degree of $v$
in $T_n$ is finite (for instance, take exactly $n$ children in total).
\end{enumerate}
\smallskip

By construction, $T_n$ has a leaf at height $|v|+1$, hence $\partial(T_n)\le |v|+1$.
On the other hand, since $T_n$ agrees with $T$ on all vertices up to height $|v|$ and
$T$ has no leaf at height $\le |v|+1$, we also have $\partial(T_n)\ge |v|+1$.
Therefore $\partial(T_n)=|v|+1$ for every $n$.
\smallskip

Finally, $T_n\to T$: for any fixed truncation at level $h$, choose $n$ large enough so
that the modification at $v$ occurs beyond the finitely many children of $v$ visible in
the truncation up to height $h$. Then the truncations of $T_n$ and $T$ up to height $h$
coincide for all $n$ large, which is exactly convergence in the local topology of
$\T_{\infty}$.
\smallskip

Conversely, assume that $\partial$ is discontinuous at $T$. Then there exists a sequence
$(T_n)_{n\ge1}$ such that $T_n\to T$ and $\partial(T_n)\not\to\partial(T)$.
Set $k:=\partial(T)$.
\smallskip

If $k<\infty$, since the set $\{\partial\ge k+1\}$ is closed in the local topology,
we have $\partial(T_n)\le k$ for all sufficiently large $n$. The failure of convergence
to $k$ then yields a subsequence with $\partial(T_n)=\ell<k$ for some integer $\ell$.
If $k=\infty$, the failure of convergence to infinity yields a bounded subsequence
of $\partial(T_n)$, and hence a further subsequence with constant finite value $\ell$.
Thus, in either case, after passing to a subsequence, we have
$\partial(T_n)=\ell<k=\partial(T)$ for all $n$.
\smallskip

For each $n$, let $u_n$ be a leaf of $T_n$ at height $\ell$, and for $0\le j\le \ell$
denote by $a_n(j)$ the ancestor of $u_n$ at height $j$ (so, in Neveu's formalism,
$a_n(0)=\varnothing$ and $a_n(\ell)=u_n$). By successive extractions, we may assume that
there exists a maximal index $j\in\{0,\dots,\ell-1\}$ such that the sequence
$(a_n(j))_{n\ge1}$ is constant; denote this common value by $v$. In particular, for every
$n$, the vertex $a_n(j+1)$ is the child of $v$ lying on the path from $v$ to $u_n$.
\smallskip 

By maximality of $j$, the sequence $(a_n(j+1))_{n\ge1}$ cannot take only finitely many
values, and hence, along the subsequence, the vertices $a_n(j+1)$ are pairwise distinct
children of $v$ in $T_n$, for infinitely many $n$.
\smallskip 

We claim that $k_v(T)=\infty$. Recall that $T_n\to T$ in $(\T_{\infty},\delta_{\infty})$
implies $k_u(T_n)\to k_u(T)$ for every node $u$: for $u=v$ and any $h$ exceeding $k_v(T)$
and the coordinates of $v$, one has $r_{h,\infty}(T_n)=r_{h,\infty}(T)$ for $n$ large, and
since the children of a node carry consecutive labels this forces $k_v(T_n)=k_v(T)$.
Suppose now, for contradiction, that $k_v(T)=M<\infty$. Then $k_v(T_n)=M$ for all $n$ large
enough, so the children of $v$ in $T_n$ are exactly $v1,\dots,vM$ and hence
$a_n(j+1)\in\{v1,\dots,vM\}$ for such $n$. Thus $(a_n(j+1))_{n}$ takes only finitely many
values along the subsequence, contradicting that it takes infinitely many distinct values.
Hence $k_v(T)=\infty$.
\smallskip 

Finally, since $\ell<\partial(T)$, the tree $T$ has no leaf at height $\le \ell$, and in
particular $\partial(T)>\ell\ge j+1=|v|+1$, which completes the proof.
\end{proof}

\begin{lem}\label{lem: partial_continuous}
Let $\P$ denote the law of the local limit $\hat{\cT}$ of $\cT_n$ on the Polish space
$(\T_{\infty},\delta_{\infty})$, for a weight sequence of type~I, II or~III. Then
\[
\P\bigl(\{T\in\T_{\infty} : \partial \text{ is discontinuous at } T \}\bigr)=0,
\]
that is, $T\mapsto\partial(T)$ is $\P$-almost surely continuous.
\end{lem}
\begin{proof}
By Proposition~\ref{Propo: discontinuity_set_delta_condensation}, every tree in
$\mathrm{Disc}(\partial)$ has a vertex $v$ of infinite outdegree with $\partial(T)>|v|+1$; in
particular $v$ has no leaf among its children. We show that this event is not charged by $\P$,
distinguishing the three regimes.
\smallskip

\emph{Type~I.} Here $m=\E(\xi)=1$, so the special vertices reproduce according to the size-biased
law $\hat{\xi}$, which has no atom at infinity, $\P(\hat{\xi}=\infty)=1-m=0$, while normal
vertices reproduce according to $\xi<\infty$. Hence every vertex of $\hat{\cT}$ has finite
outdegree, i.e.\ $\hat{\cT}$ is almost surely locally finite, and therefore
$\P\bigl(\exists v: k_v(\hat{\cT})=\infty\bigr)=0$. In particular $\P(\mathrm{Disc}(\partial))=0$.
\smallskip

\emph{Type~II.} Let $v$ be the almost surely unique vertex of infinite outdegree. Conditionally
on the $\sigma$-field $\mathcal F_{|v|}$ generated by the tree up to generation $|v|$, the
subtrees rooted at the children of $v$ are i.i.d.\ Bienaym\'e--Galton--Watson trees with
offspring distribution $\xi$. Writing $A_n$ for the event that the first $n$ children of $v$ are
not leaves and $A$ for the event that no child of $v$ is a leaf,
\[
\P(A_n\mid\mathcal F_{|v|})=(1-p_0)^n,\qquad p_0=\P(\xi=0)>0,
\]
since $\xi$ is subcritical. Taking expectations, $\P(A)\le\P(A_n)=(1-p_0)^n\to0$, so $\P(A)=0$
and $\P(\mathrm{Disc}(\partial))=0$.
\smallskip

\emph{Type~III.} The limit degenerates: the condensation vertex is the root, all of whose
children are leaves, so $\partial(\hat{\cT})=1$ almost surely and the root has a leaf among its
children with probability one. Hence $\mathrm{Disc}(\partial)$ is not charged.
\end{proof}

\begin{theo}\label{theo: conv_dist_partial}
Let $\cT_n$ be a size conditioned simply generated tree of type~I, II or~III, and let
$\hat{\cT}$ be the corresponding local limit around the root. Then
\[
\partial(\cT_n) \xrightarrow{d} \partial(\hat{\cT})\,,\quad \text{as $n\to\infty$,}
\]
and, for every $k\ge1$,
\[
\lim_{\substack{n \to \infty}}\P\bigl(\partial(\cT_n) \ge k\bigr)
=\P\bigl(\partial(\hat{\cT}) \ge k\bigr).
\]
\end{theo}
\begin{proof}
By Janson~\cite[Thm.~7.1, p.~121]{Janson2012}, $\cT_n\xrightarrow{d}\hat{\cT}$ in the local
topology of $(\T_{\infty},\delta_{\infty})$. For every $k\ge0$ the set
$\{T\in\T_{\infty}:\partial(T)\ge k\}$ is closed, so $\partial$ is Borel measurable, and by
Lemma~\ref{lem: partial_continuous} the functional $\partial$ is continuous almost surely under
the law of $\hat{\cT}$. The mapping theorem~\cite[Thm.~2.7]{MR1700749} then yields
$\partial(\cT_n)\xrightarrow{d}\partial(\hat{\cT})$, and the displayed identity for the tails
follows since $\{\partial\ge k\}$ is a continuity set. In the critical case (type~I) one may
argue more directly: by Lemma~\ref{lem: clopen} the set $\{\partial\ge k\}$ is clopen in
$(\T,\delta)$, hence a continuity set, and the portmanteau theorem applies.
\end{proof}

\begin{remark}\label{rem:type_III}
In type~III the limiting object degenerates: the condensation tree $\hat{\cT}$ has a root of
infinite outdegree all of whose children are leaves, so $\partial(\hat{\cT})=1$ almost surely.
By the local limit in this regime (\cite[\S4]{Janson2012}, \cite[\S5]{zbMATH07039768}) we
have $\partial(\cT_n)\xrightarrow{d}1$.
\finremark
\end{remark}

\subsection{Exponential tail bound and uniform integrability}\label{sec: distance_borde_UI}
In this subsection we prove that the sequence of random variables \(\partial(\cT_n)\) is uniformly integrable of every order. Consequently, the moments of \(\partial(\cT_n)\) converge to those of \(\partial(\hat{\cT})\).
\medskip

The main result of this subsection is Proposition~\ref{lem: sub_exp_critical}: a uniform exponential tail bound for the distance from the root to the nearest leaf, valid without any assumption on the weight types. We prove it through six auxiliary lemmas.

\medskip
The strategy is to bound the slightly larger quantity $\partial^{*}(\cT_n)\ge\partial(\cT_n)$, defined as the depth of the leftmost leaf in depth-first order. The advantage of $\partial^{*}$ is that the event $\{\partial^{*}\ge k\}$ has a clean combinatorial description in terms of the DFS outdegree sequence, which lets us recast the problem as a question about a uniformly random ordering of a fixed multiset. 
\medskip

The lemmas progress as follows. Lemma~\ref{lem: dfs_uniform_Luk} identifies the conditional law of the DFS sequence given the multiset of outdegrees as the uniform law on the corresponding \L{}ukasiewicz orderings. Lemma~\ref{lem: exch_repr} shows that this constrained uniform law can be recovered from an unconstrained uniform law on \emph{all} orderings, via a cyclic shift conditioned on a \L{}ukasiewicz event $\mathcal E_n$. Lemma~\ref{lem: prefix_sampling} computes the law of the first $k$ coordinates of the unconstrained ordering as a uniform sample without replacement. Lemma~\ref{lem: cond_cycle} computes the conditional probability of $\mathcal E_n$ given the first $k$ values via a refinement of the cycle lemma. Lemma~\ref{lem: master_identity} combines the previous lemmas into an exact formula for $\P(\partial^{*}(\cT_n)\ge k\mid\mathcal G_n)$. Lemma~\ref{lem: cond_bound} extracts an explicit upper bound from this formula in terms of the proportion of leaves $\eta=L/n$. The proof of Proposition~\ref{lem: sub_exp_critical} then takes expectations and splits according to whether $\eta$ is close to its limiting value $p_0=\P(\xi=0)$, the asymptotic proportion of leaves, or not, using Janson's universal leaf-concentration estimate to handle the atypical regime.

\medskip
We now turn to the setup. Let $\cT_n$ be the simply generated tree of size $n$, and write $\P$ and $\E$ for its law and the corresponding expectation; in types~I and~II, $\cT_n$ is the Bienaym\'e--Galton--Watson tree with offspring distribution $\xi$ conditioned on having $n$ vertices. Denote by $(d^{\mathrm{dfs}}_1,\dots,d^{\mathrm{dfs}}_n)$ the outdegree sequence of $\cT_n$ in depth-first order, and let
\[
L:=n_{\cT_n}(0)
\]
be the number of leaves; for $n\ge 2$, $1\le L\le n-1$. Write $M(\cT_n):=(n_{\cT_n}(j))_{j\ge 0}$ for the (random) multiset of outdegrees of $\cT_n$, and let $\mathcal G_n:=\sigma(M(\cT_n))$ be the $\sigma$-algebra it generates. The leftmost-leaf depth $\partial^{*}(\cT_n)$ admits the following description in terms of the DFS sequence: $\partial(\cT_n)\le\partial^{*}(\cT_n)$ pointwise, and
\[
\{\partial^{*}(\cT_n)\ge k\}=\{d^{\mathrm{dfs}}_1>0,\dots,d^{\mathrm{dfs}}_k>0\}.
\]

The previous equality is illustrated by the following figure.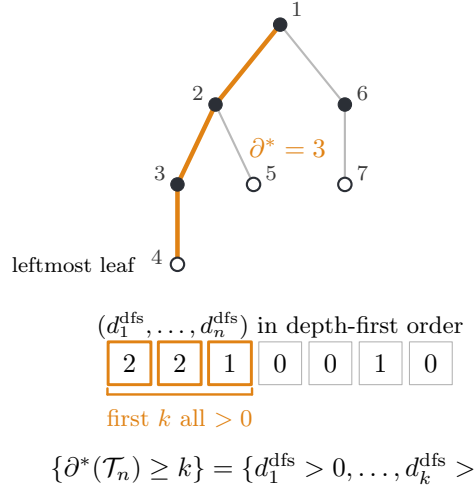
\begin{figure}[H]
\centering
\resizebox{0.48\linewidth}{!}{%
\begin{tikzpicture}[line cap=round,line join=round,
  int/.style ={circle,draw=cInt,fill=cInt,minimum size=5pt,inner sep=0},
  leaf/.style={circle,draw=cInt,fill=white,line width=.8pt,minimum size=5pt,inner sep=0},
  ed/.style  ={gray!55,line width=0.8pt},
  near/.style={pRoot,line width=1.7pt},
  idx/.style ={font=\scriptsize,gray!45!black},
  cell/.style={draw=gray!60,minimum size=16pt,inner sep=0,font=\small},
  cellhi/.style={draw=pRoot,line width=1.1pt,minimum size=16pt,inner sep=0,font=\small}]
\coordinate (v1) at (0,0);
\coordinate (v2) at (-0.85,-1.05); \coordinate (v6) at (0.85,-1.05);
\coordinate (v3) at (-1.35,-2.1);  \coordinate (v5) at (-0.35,-2.1); \coordinate (v7) at (0.85,-2.1);
\coordinate (v4) at (-1.35,-3.15);
\draw[ed] (v1)--(v6); \draw[ed] (v2)--(v5); \draw[ed] (v6)--(v7);
\draw[ed] (v1)--(v2)--(v3)--(v4);
\draw[near] (v1)--(v2)--(v3)--(v4);
\node[int] at (v1){}; \node[int] at (v2){}; \node[int] at (v3){}; \node[int] at (v6){};
\node[leaf] at (v4){}; \node[leaf] at (v5){}; \node[leaf] at (v7){};
\node[idx] at (0.22,0.2) {$1$};   \node[idx] at (-1.08,-0.9) {$2$};
\node[idx] at (-1.58,-1.95){$3$}; \node[idx] at (-1.62,-3.05){$4$};
\node[idx] at (-0.12,-1.95){$5$}; \node[idx] at (1.08,-0.9){$6$}; \node[idx] at (1.08,-1.95){$7$};
\node[pRoot,font=\small,anchor=west] at (-0.72,-1.62) {\, $\partial^{*}=3$};
\node[font=\scriptsize,anchor=east] at (-1.78,-3.15) {leftmost leaf};
\node[font=\footnotesize] at (0,-3.95) {$(d^{\mathrm{dfs}}_1,\dots,d^{\mathrm{dfs}}_n)$ in depth-first order};
\node[cellhi] (c1) at (-1.98,-4.45) {$2$};
\node[cellhi] (c2) at (-1.32,-4.45) {$2$};
\node[cellhi] (c3) at (-0.66,-4.45) {$1$};
\node[cell]   (c4) at (0,-4.45) {$0$};
\node[cell]   (c5) at (0.66,-4.45) {$0$};
\node[cell]   (c6) at (1.32,-4.45) {$1$};
\node[cell]   (c7) at (1.98,-4.45) {$0$};
\draw[pRoot,line width=0.9pt] (-2.27,-4.86)--(-0.37,-4.86);
\draw[pRoot,line width=0.9pt] (-2.27,-4.86)--(-2.27,-4.78);
\draw[pRoot,line width=0.9pt] (-0.37,-4.86)--(-0.37,-4.78);
\node[pRoot,font=\footnotesize,anchor=north] at (-1.32,-4.92) {first $k$ all $>0$};
\node[font=\small,anchor=north] at (0,-5.5)
  {$\{\partial^{*}(\cT_n)\ge k\}=\{d^{\mathrm{dfs}}_1>0,\dots,d^{\mathrm{dfs}}_k>0\}$};
\end{tikzpicture}%
}
\caption{The leftmost-leaf bound behind the exponential tail estimate.}
\label{fig:dfs_leftmost}
\end{figure}

\smallskip
For a deterministic multiset $M$ of non-negative integers with $\sum_c M(c)=n$ and $\sum_c c\,M(c)=n-1$, write $\Sigma_M$ for its set of distinguishable orderings, and $\Sigma_M^{\mathrm{Luk}}\subset\Sigma_M$ for the subset of \emph{\L{}ukasiewicz orderings}, that is, sequences $(a_1,\dots,a_n)\in\Sigma_M$ satisfying $\sum_{i=1}^k(a_i-1)\ge 0$ for every $k=1,\dots,n-1$.

\medskip
Our first lemma identifies the conditional law of the DFS sequence given the multiset of outdegrees. It is a direct consequence of two well-known facts: that the simply generated weight of a tree depends only on its multiset of outdegrees, and that the DFS map gives a bijection between plane trees of size $n$ and \L{}ukasiewicz sequences of length $n$.

\begin{lem}\label{lem: dfs_uniform_Luk}
Conditional on $\mathcal G_n$, the DFS outdegree sequence $(d^{\mathrm{dfs}}_1,\dots,d^{\mathrm{dfs}}_n)$ is uniformly distributed on $\Sigma_{M(\cT_n)}^{\mathrm{Luk}}$.
\end{lem}

\begin{proof}
Since the simply generated weight of a plane tree $t$ with $n$ vertices is $\prod_{v\in t}\omega_{\deg^+(v)}$, it depends on $t$ only through its multiset of outdegrees. Conditioning on $\mathcal G_n$ therefore makes $\cT_n$ uniform on the (finite) set of plane trees with multiset $M(\cT_n)$. The DFS map is a bijection between plane trees of size $n$ and \L{}ukasiewicz sequences of length $n$, and it sends a tree with multiset $M$ to a sequence in $\Sigma_M^{\mathrm{Luk}}$. The uniform law on plane trees with multiset $M(\cT_n)$ is therefore mapped bijectively to the uniform law on $\Sigma_{M(\cT_n)}^{\mathrm{Luk}}$.
\end{proof}

The constraint that the DFS sequence be \L{}ukasiewicz is a global one, it involves all partial sums, and is not easy to manage for direct probabilistic computations. The next lemma circumvents this by lifting the law to the unconstrained set $\Sigma_M$ of all orderings. Given a fixed multiset $M$, define on a possibly enlarged probability space, conditional on $\{M(\cT_n)=M\}$,
\[
(\tilde d_1,\dots,\tilde d_n)\text{ uniform on }\Sigma_M,\qquad U\text{ uniform on }\{1,\dots,n\}\text{ independent of }\tilde d,
\]
and the cyclic shift
\[
\tilde d^{(U)}_i:=\tilde d_{((U+i-1-1)\bmod n)+1},\qquad i=1,\dots,n,
\]
together with the \L{}ukasiewicz event
\[
\mathcal E_n:=\Bigl\{\textstyle\sum_{i=1}^m(\tilde d^{(U)}_i-1)\ge 0\,,\ \forall\,m=1,\dots,n-1\Bigr\}.
\]
Throughout the rest of the section, we also write
\[
A_k:=\{\tilde d^{(U)}_1>0,\dots,\tilde d^{(U)}_k>0\},\qquad
H:=\sum_{i=1}^k(\tilde d^{(U)}_i-1).
\]
On $A_k$, each summand of $H$ is nonnegative, so $H\ge 0$.

\begin{lem}\label{lem: exch_repr}
Conditional on $\mathcal G_n$:
\begin{enumerate}[label=\textup{(\roman*)}]
\item $\tilde d^{(U)}$ is uniform on $\Sigma_{M(\cT_n)}$;
\item $\P(\mathcal E_n\mid\mathcal G_n)=1/n$;
\item $(d^{\mathrm{dfs}}_1,\dots,d^{\mathrm{dfs}}_n)\stackrel{d}{=} \left((\tilde d^{(U)}_1,\dots,\tilde d^{(U)}_n)\,\bigm|\,\mathcal E_n\right)$.
\end{enumerate}
\end{lem}

\begin{proof}
Fix a realisation $M$ of $M(\cT_n)$.

\smallskip
For (i), let $\tau\in\Sigma_M$. The equation $\rho^{j-1}\sigma=\tau$, where $\rho$ denotes cyclic left-shift, has the unique solution $\sigma=\rho^{-(j-1)}\tau$ for each $j\in\{1,\dots,n\}$. Hence $\{\tilde d^{(U)}=\tau\}$ has exactly $n$ preimages in $\Sigma_M\times\{1,\dots,n\}$, each of probability $1/(n|\Sigma_M|)$, so $\P(\tilde d^{(U)}=\tau\mid M(\cT_n)=M)=1/|\Sigma_M|$.

\smallskip
For (ii), we first check that the cyclic group $\Z/n\Z$ acts freely on $\Sigma_M$: if some $\sigma\in\Sigma_M$ had nontrivial period (i.e., there exists \(p<n\) such that the vector is
obtained by repeating a block of length \(p\)) therefore $p\mid n$ with $p<n$, then $n-1=\sum_i\sigma_i=(n/p)\sum_{i\le p}\sigma_i$ would force $n\mid p(n-1)$ and hence $n\mid p$ (since $\gcd(n,n-1)=1$), a contradiction. Every orbit therefore has size exactly $n$, and by the cycle lemma applied with $r=1$, exactly one element of each orbit is \L{}ukasiewicz. This gives $|\Sigma_M^{\mathrm{Luk}}|=|\Sigma_M|/n$, and combined with (i),
\[
\P(\mathcal E_n\mid M(\cT_n)=M)=\frac{|\Sigma_M^{\mathrm{Luk}}|}{|\Sigma_M|}=\frac{1}{n}.
\]

\smallskip
For (iii), let $\sigma\in\Sigma_M^{\mathrm{Luk}}$. By (i) and (ii),
\[
\P(\tilde d^{(U)}=\sigma\mid\mathcal E_n,M(\cT_n)=M)
=\frac{\P(\tilde d^{(U)}=\sigma\mid M(\cT_n)=M)}{\P(\mathcal E_n\mid M(\cT_n)=M)}
=\frac{1/|\Sigma_M|}{1/n}=\frac{1}{|\Sigma_M^{\mathrm{Luk}}|}.
\]
The right-hand side does not depend on $\sigma\in\Sigma_M^{\mathrm{Luk}}$, so $\tilde d^{(U)}\mid\mathcal E_n,M(\cT_n)=M$ is uniform on $\Sigma_M^{\mathrm{Luk}}$. By Lemma~\ref{lem: dfs_uniform_Luk}, this matches the conditional law of the DFS sequence.
\end{proof}

\medskip
Once the DFS sequence is recast as $\tilde d^{(U)}\mid\mathcal E_n$ via Lemma~\ref{lem: exch_repr}, the prefix law before imposing $\mathcal E_n$ is just sampling without replacement, which factors over coordinates and allows direct expectation calculations.

\begin{lem}\label{lem: prefix_sampling}
Conditional on $\mathcal G_n$, the prefix $(\tilde d^{(U)}_1,\dots,\tilde d^{(U)}_k)$ is a uniform ordered sample without replacement of size $k$ from the multiset $M(\cT_n)$. In particular,
\begin{equation}\label{eq: Ak_prob}
\P(A_k\mid\mathcal G_n)=\frac{(n-L)_k}{(n)_k},\qquad k=1,\dots,n,
\end{equation}
where $(m)_j=m(m-1)\cdots(m-j+1)$.
\end{lem}

\begin{proof}
By Lemma~\ref{lem: exch_repr}(i), conditional on $\{M(\cT_n)=M\}$, $\tilde d^{(U)}$ is uniform on $\Sigma_M$. For any sequence $(b_1,\dots,b_k)\in\Z_{\ge 0}^k$ that can be extracted from $M$,
\begin{align*}
\P\bigl((\tilde d^{(U)}_1,\dots,\tilde d^{(U)}_k)=(b_1,\dots,b_k)\bigm|M(\cT_n)=M\bigr)
&=\frac{\#\{\sigma\in\Sigma_M:\sigma_i=b_i,\,i\le k\}}{|\Sigma_M|} \\
&=\frac{(n-k)!\prod_c M(c)!}{n!\prod_c M_{\rm rest}(c)!},
\end{align*}
where $M_{\rm rest}(c):=M(c)-\#\{i\le k:b_i=c\}$. This is exactly the law of an ordered draw of size $k$ without replacement from the multiset $M$, proving the first claim. Equation~\eqref{eq: Ak_prob} is the probability that the $k$ draws all come from the $n-L$ positive outdegrees among $n$ total balls.
\end{proof}

\medskip
We now strengthen Lemma~\ref{lem: exch_repr} (ii) by computing the conditional probability of $\mathcal E_n$ once the first $k$ values of $\tilde d^{(U)}$ are known. The proof reduces this to the cycle lemma applied to the residual sequence after the first $k$ steps, this time with $r=h+1$ rather than $r=1$.

\begin{lem}\label{lem: cond_cycle}
For every $1\le k\le n-1$, on $A_k$,
\[
\P\bigl(\mathcal E_n\,\bigm|\,\mathcal G_n,\tilde d^{(U)}_1,\dots,\tilde d^{(U)}_k\bigr)=\frac{H+1}{n-k}.
\]
\end{lem}

\begin{proof}
Fix a realisation $M$ of $M(\cT_n)$ and condition on the event
\[
\mathcal F:=\bigl\{M(\cT_n)=M,\,\tilde d^{(U)}_1=a_1,\dots,\tilde d^{(U)}_k=a_k\bigr\},
\]
where $a_i\ge 1$ for each $i$. Set $h:=\sum_{i=1}^k(a_i-1)\ge 0$.

\smallskip
By Lemma~\ref{lem: prefix_sampling}, conditional on $\mathcal F$, the residual $(\tilde d^{(U)}_{k+1},\dots,\tilde d^{(U)}_n)$ is uniform on the orderings of the complementary multiset $M':=M\setminus\setminus\{a_1,\dots,a_k\}$ of size $n-k$, with $\sum_{a\in M'}(a-1)=-1-h$. Setting $W_m:=\sum_{j=1}^m(\tilde d^{(U)}_{k+j}-1)$, the event $\mathcal E_n$ on $A_k$ is equivalent to
\[
G:=\{W_m\ge -h\ \forall\,m=1,\dots,n-k-1\}.
\]

\smallskip
We now apply the cycle lemma to the residual. Attach distinguishable labels $\Lambda=\{\ell_1,\dots,\ell_{n-k}\}$ to the elements of $M'$, with value map $\delta:\Lambda\to\Z_{\ge 0}$ realising $M'$ as a multiset, and let $\Pi(\Lambda):=\{\sigma:\{1,\dots,n-k\}\to\Lambda\text{ bijection}\}$. The cyclic shift $\rho(\sigma)=(\sigma(2),\dots,\sigma(n-k),\sigma(1))$ generates a free $\Z/(n-k)\Z$ action on $\Pi(\Lambda)$. The cycle lemma applied to the integer sequence $(\delta(\sigma(1))-1,\dots,\delta(\sigma(n-k))-1)$, of length $n-k$, increments $\ge -1$ and total sum $-(h+1)$, that is, with $r=h+1$, yields that exactly $h+1$ of the $n-k$ rotations produce a value sequence in $G$. Hence, deterministically,
\[
N(\sigma):=\#\{j\in\{0,\dots,n-k-1\}:\rho^j\sigma\in G\}=h+1\qquad\text{for all }\sigma\in\Pi(\Lambda).
\]

\smallskip
Now let $\sigma$ be a uniformly random element of $\Pi(\Lambda)$, the labelled lift of the conditional law of $(\tilde d^{(U)}_{k+1},\dots,\tilde d^{(U)}_n)$ under $\mathcal F$. By rotation-invariance of the law of $\sigma$ on $\Pi(\Lambda)$,
\[
(n-k)\,\P(\sigma\in G\mid\mathcal F)
=\sum_{j=0}^{n-k-1}\P(\rho^j\sigma\in G\mid\mathcal F)
=\E[N(\sigma)\mid\mathcal F]=h+1.
\]
Since this holds for every realisation $M$ of $M(\cT_n)$ and every $(a_1,\dots,a_k)$ with $a_i\ge 1$, the claim follows.
\end{proof}

\medskip
Lemmas \ref{lem: dfs_uniform_Luk}--\ref{lem: cond_cycle} combine into an exact formula for the conditional tail probability of $\partial^{*}(\cT_n)$.

\begin{lem}\label{lem: master_identity}
For every $1\le k\le n-1$, almost surely,
\[
\P(\partial^{*}(\cT_n)\ge k\mid\mathcal G_n)
=\frac{n}{n-k}\Bigl(k\cdot\tfrac{L-1}{n-L}+1\Bigr)\frac{(n-L)_k}{(n)_k},
\]
with the convention $(n-L)_k/(n)_k=0$ if $k>n-L$.
\end{lem}

\begin{proof}
By Lemma~\ref{lem: dfs_uniform_Luk} and the DFS interpretation of $\{\partial^{*}(\cT_n)\ge k\}$,
\[
\P(\partial^{*}(\cT_n)\ge k\mid\mathcal G_n)=\P(A_k\mid\mathcal E_n,\mathcal G_n).
\]
Using Lemma~\ref{lem: exch_repr} (ii) for the denominator and Lemma~\ref{lem: cond_cycle} for the numerator,
\begin{align*}
\P(A_k\mid\mathcal E_n,\mathcal G_n)
&=\frac{\P(A_k\cap\mathcal E_n\mid\mathcal G_n)}{1/n}
=n\,\E\bigl[1_{A_k}\,\P(\mathcal E_n\mid\mathcal G_n,\tilde d^{(U)}_1,\dots,\tilde d^{(U)}_k)\,\bigm|\,\mathcal G_n\bigr]\\
&=\frac{n}{n-k}\,\E\bigl[(H+1)1_{A_k}\mid\mathcal G_n\bigr],
\end{align*}
since $1_{A_k}=0$ off $A_k$. It remains to evaluate the last expectation.

\smallskip
By Lemma~\ref{lem: prefix_sampling}, conditional on $A_k \cap \mathcal G_n$ the prefix $(\tilde d^{(U)}_1,\dots,\tilde d^{(U)}_k)$ is a uniform ordered sample without replacement from the $n-L$ positive outdegrees of $\cT_n$. Since the sum of these positive outdegrees equals $n-1$, exchangeability gives
\[
\E[\tilde d^{(U)}_i\mid A_k,\mathcal G_n]=\frac{n-1}{n-L}\quad\text{for every }i\le k,
\quad\text{so}\quad
\E[H\mid A_k,\mathcal G_n]=k\cdot\frac{L-1}{n-L}.
\]
Combining with $\P(A_k\mid\mathcal G_n)=(n-L)_k/(n)_k$ from~\eqref{eq: Ak_prob},
\[
\E[(H+1)1_{A_k}\mid\mathcal G_n]=\Bigl(k\cdot\tfrac{L-1}{n-L}+1\Bigr)\frac{(n-L)_k}{(n)_k}.
\]
This concludes the proof. 
\end{proof}

\medskip
The identity in Lemma~\ref{lem: master_identity} is exact but somewhat opaque. Bounding each of its three factors crudely yields a more transparent expression in terms of the proportion of leaves $\eta=L/n$ alone, which is the form we will actually use.

\begin{lem}\label{lem: cond_bound}
Set $\eta:=L/n\in(0,1)$. For every $1\le k\le n/2$,
\[
\P(\partial^{*}(\cT_n)\ge k\mid\mathcal G_n)
\le 2\Bigl(\tfrac{k\eta}{1-\eta}+1\Bigr)(1-\eta)^k
=2k\eta(1-\eta)^{k-1}+2(1-\eta)^k.
\]
\end{lem}

\begin{proof}
Start from Lemma~\ref{lem: master_identity} and bound each factor:
$n/(n-k)\le 2$ for $k\le n/2$;
$(L-1)/(n-L)\le L/(n-L)=\eta/(1-\eta)$;
and $(n-L)_k/(n)_k\le(1-\eta)^k$, since
\[
\prod_{i=0}^{k-1}\frac{n-L-i}{n-i}\le\Bigl(\frac{n-L}{n}\Bigr)^k=(1-\eta)^k
\]
because $i\mapsto(n-L-i)/(n-i) = 1-L/(n-i)$ is non-increasing.
\end{proof}
\smallskip

With Lemmas \ref{lem: dfs_uniform_Luk}--\ref{lem: cond_bound} in place, the proof of the proposition reduces to taking expectations of the bound in Lemma~\ref{lem: cond_bound} and controlling the resulting integrand on the two regimes $\eta\ge c_0$ and $\eta<c_0$. The first regime gives genuine exponential decay in $k$ from the factors $(1-\eta)^k$; the second is rare by Janson's leaf-concentration estimate and contributes a term that can be absorbed for $n$ large.

\begin{propo}\label{lem: sub_exp_critical}
Let $(\cT_n)_{n\geq1}$ be a sequence of size-conditioned simply generated trees of type~I, II or~III. Then there are constants $n_0\ge 1$ and $C,c>0$,
depending only on the weight sequence, such that\[
\P(\partial(\cT_n)\ge k)\le Ce^{-ck}\,, \quad \text{ for any } n \geq n_0 \text{ and any } k \geq 0.
\]
\end{propo}
\smallskip

\begin{proof}
Since $\partial(\cT_n)\le\partial^{*}(\cT_n)$ pointwise, it suffices to prove the bound for $\partial^{*}(\cT_n)$.

Fix $c_0\in(0,p_0)$. By Lemma~\ref{lem: cond_bound}, for $1\le k\le n/2$,
\begin{equation}\label{eq: tail_two_terms}
\P(\partial^{*}(\cT_n)\ge k\mid\mathcal G_n)\le 2k\eta(1-\eta)^{k-1}+2(1-\eta)^k,\qquad \eta=L/n.
\end{equation}
We bound the expectation of \eqref{eq: tail_two_terms} by splitting according to whether $\eta\ge c_0$ or $\eta<c_0$.

\smallskip
\emph{Regime $\eta\ge c_0$.} Then $\eta(1-\eta)^{k-1}\le(1-c_0)^{k-1}$ and $(1-\eta)^k\le(1-c_0)^k$, so the right-hand side of \eqref{eq: tail_two_terms} is $\le 2(k+1)(1-c_0)^{k-1}$. Setting $c_1:=-\log(1-c_0)>0$ and using that $x\mapsto(x+1)e^{-c_1 x/2}$ is bounded on $[0,\infty)$, there exist constants $C_2,c_2>0$ (depending only on $c_0$) with
\[
\E\left[\bigl(2k\eta(1-\eta)^{k-1}+2(1-\eta)^k\bigr)1_{\{\eta\ge c_0\}}\right]\le C_2 e^{-c_2 k}.
\]

\emph{Regime $\eta<c_0$.} Pointwise on $\{1\le L\le n-1\}$, $2k\eta(1-\eta)^{k-1}+2(1-\eta)^k\le 2k+2\le 4k$. By Janson's leaf concentration for conditioned Galton--Watson trees (\cite[Theorem~11.4]{Janson2012}, restated as \cite[Theorem~5.1]{AddarioBerryBrandenbergerHamdanKerriou2022}), applied with $\varepsilon:=p_0-c_0>0$, there exist $a>0$ and $n_0\ge 1$ such that for every $n\ge n_0$,
\[
\P(L/n<c_0)\le\P(|L/n-p_0|>\varepsilon)\le e^{-an}.
\]
Hence
\[
\E\left[\bigl(2k\eta(1-\eta)^{k-1}+2(1-\eta)^k\bigr)1_{\{\eta<c_0\}}\right]\le 4k\,e^{-an}\le 4n\,e^{-an}.
\]
\smallskip

For $1\le k\le n/2$ and $n\ge n_0$,
\[
\P(\partial^{*}(\cT_n)\ge k)\le C_2 e^{-c_2 k}+4n\,e^{-an}.
\]
For $n$ large enough, $4n e^{-an/2}\le 1$, so $4n e^{-an}\le e^{-an/2}\le e^{-ak}$ (since $n\ge 2k$). Hence
\[
\P(\partial^{*}(\cT_n)\ge k)\le C_3 e^{-c_3 k},\qquad 1\le k\le n/2,
\]
with $C_3:=C_2+1$ and $c_3:=\min(c_2,a)>0$.
\smallskip

For $k=0$, the trivial bound $\P(\partial^{*}(\cT_n)\ge 0)=1$ is dominated by any $C\ge 1$. For $n/2<k\le n$, monotonicity in $k$ and the previous step applied at $k_0:=\lfloor n/2\rfloor\ge k/2-1/2$ give
\[
\P(\partial^{*}(\cT_n)\ge k)\le C_3 e^{-c_3 k_0}\le C_3 e^{c_3/2}e^{-(c_3/2)k}.
\]

Setting $C:=\max(1,C_3 e^{c_3/2})$ and $c:=c_3/2>0$ yields $\P(\partial^{*}(\cT_n)\ge k)\le C e^{-ck}$ for every $n\ge n_0$ and every $k\ge 0$, and the bound transfers to $\partial(\cT_n)$ through the inequality $\partial(\cT_n) \leq \partial^{*}(\cT_n)$, which holds almost surely. 
\end{proof}

Using the previous tail bound, we obtain uniform integrability.

\begin{lem}\label{lem: integrabilidad} Let $(\cT_n)_{n\geq1}$ be a sequence of size-conditioned simply generated trees of type~I, II or~III. Then the family of random variables
$\{\partial(\cT_n)\}_{n\ge1}$ is uniformly integrable of every order.
\end{lem}
        
		\begin{proof} Fix an integer \(k \geq 1\). Lemma \ref{lem: sub_exp_critical} gives that the probability of the right-tail of $\partial(\cT_n)$ is exponentially bounded. Applying this lemma there are constants $n_0 \geq 1$ and \(C, c > 0\) (only depending on the offspring distribution) such that
			\[
			\P(\partial(\cT_n) \geq k) \leq C e^{-ck}\,, \quad \text{ for any } n \geq n_0 \text{ and } k \geq 0\,.
			\]
			
			Fix a real number $p \ge0$. For any $n\ge n_0$,
\[
\E\bigl[\partial(\cT_n)^{p}\bigr]
\le \sum_{k\ge1} k^{p}\,\P(\partial(\cT_n)\ge k)
\le C\sum_{k\ge1} k^{p} e^{-ck}<+\infty,
\]
so in particular $\sup_{n\ge n_0}\E[\partial(\cT_n)^{p}]<+\infty$. A uniform moment bound does
not by itself yield uniform integrability, so we estimate the tail contribution directly.
For every $K\ge0$, using $\P(\partial(\cT_n)=j)\le\P(\partial(\cT_n)\ge j)$,
\[
\sup_{n\ge n_0}\E\bigl[\partial(\cT_n)^{p}\,\mathbf 1_{\{\partial(\cT_n)>K\}}\bigr]
\le \sum_{j>K} j^{p}\,\sup_{n\ge n_0}\P(\partial(\cT_n)\ge j)
\le C\sum_{j>K} j^{p} e^{-cj}.
\]
The right-hand side does not depend on $n$ and, since $\sum_{j\ge1} j^{p} e^{-cj}<\infty$,
tends to $0$ as $K\to\infty$. Hence
\[
\lim_{K\to\infty}\ \sup_{n\ge n_0}\E\bigl[\partial(\cT_n)^{p}\,\mathbf 1_{\{\partial(\cT_n)>K\}}\bigr]=0,
\]
which is precisely uniform integrability of $\{\partial(\cT_n)^{p}\}_{n\ge n_0}$. Since
$p\ge0$ was arbitrary, $\{\partial(\cT_n)\}_{n\ge1}$ is uniformly integrable of every order.
		\end{proof}

\begin{theo} Let $(\cT_n)_{n\ge1}$ be a sequence of size-conditioned simply generated trees of
type~I, II or~III. Then, for types I and II, and for every $p>0$,
\[
\lim_{\substack{n \to \infty}}
\E\bigl[\partial(\cT_n)^{p}\bigr]
=\E\bigl[\partial(\hat{\cT})^{p}\bigr]
=
\displaystyle\sum_{k=1}^{\infty}\bigl(k^{p}-(k-1)^{p}\bigr)c_k,
\]
In type~III the limit is degenerate: $\partial(\hat{\cT})=1$ almost surely, so every moment of the limit equals~$1$.
\end{theo}
\begin{proof} The result follows by combining Theorems \ref{theo: conv_dist_partial}, that is, convergence in distribution with the uniform integrability of any order of the sequence of random variables \(\{\partial(\cT_n)\}_{n\geq 1}\), see Lemma \ref{lem: integrabilidad} above. These two properties together imply convergence of the moments, see, for instance, Theorem 3.5 of \cite[p. 30]{MR1700749}.
\end{proof}

\subsection{Calculation of limiting probabilities}

Let $G(z) = \sum_{j = 0}^{\infty} p_j z^j$, for $|z| \leq 1$, be the probability generating function of $\xi$. Additionally, we define $F(z) = G(z) - p_0$, and for any $j \geq 1$, we consider
\begin{align*}
	F_j(z) = F(F_{j-1}(z))\,,\quad \text{ for any } |z|\le1,
\end{align*}
the $j$-th iterate of $F(z)$ with the convention that $F_0(z) = z$.
\smallskip

We now compute the tail probabilities of $\partial(\hat{\cT})$. Because the construction of $\hat{\cT}$ is the same in all three regimes, a single recurrence settles the three types at once; the regime enters only through the value $m = G'(1) = \E(\xi)$.

\begin{propo}\label{lem: c_k}
	Let $(\cT_n)_{n\geq1}$ be a sequence of size-conditioned simply generated trees of type~I, II or~III, and set $c_k := \P(\partial(\hat{\cT}) \geq k)$. Then $c_0 = c_1 = 1$ and, for every $k \geq 1$,
	\begin{align}\label{eq: describe_c_k}
		c_k = \prod_{j=1}^{k-1} G'(F_j(1))\,.
	\end{align}
	For types~I and~II this admits the closed form
	\begin{align}\label{eq: closed_c_k}
		c_k = {F_k'(1)}/{m}\,, \qquad \text{ for every } k \geq 1\,,
	\end{align}
	so that $c_k = F_k'(1)$ when $m = 1$ (type~I) and $c_k = F_k'(1)/m$ when $0 < m < 1$ (type~II). In type~III the limit $\hat{\cT}$ is the infinite star, and hence $c_k = 1$, for any $k \leq 1$ and $c_k = 0$, for any $k \geq 2$.
\end{propo}
\begin{proof}
	We compute the probability that $\hat{\cT}$ has no leaf before level $k$, that is, that $\partial(\hat{\cT}) \geq k$. The root is special, so it is never a leaf and $c_0 = c_1 = 1$. Fix $k \geq 2$ and condition on the offspring of the root. Among the children of a special vertex with finitely many children, exactly one, chosen uniformly at random, is special and roots an independent copy of $\hat{\cT}$, while the remaining children are normal and root independent copies of $\cT$. For $\partial(\hat{\cT}) \geq k$ the special child must satisfy $\partial \geq k-1$ and each normal child must satisfy $\partial(\cT) \geq k-1$. Writing $Z_1(\hat{\cT})$ for the number of children of the root, and using that the special child is equally likely to be any of the $j$ offspring (so the factor $1/j$ cancels the $j$ in $\P(\hat{\xi}=j) = j p_j$), we obtain
	\begin{align}\label{eq: recurrence}
		\begin{split}
			\P(\partial(\hat{\cT}) \geq k)
			&= \sum_{j=1}^{\infty} \P(\hat{\xi} = j)\, \P(\partial(\hat{\cT}) \geq k-1)\, \P(\partial(\cT) \geq k-1)^{\,j-1} \\[.1cm]
			&= G'\bigl(\P(\partial(\cT) \geq k-1)\bigr)\, \P(\partial(\hat{\cT}) \geq k-1)\,,
		\end{split}
	\end{align}
	where we used $\sum_{j \geq 1} j p_j\, z^{j-1} = G'(z)$, for any $|z| \leq 1$. If the root has infinitely many children (possible in types~II and~III), then almost surely one of the normal copies of $\cT$ is a single leaf, so $\partial(\hat{\cT}) = 1$; this is consistent with the atom $\P(\hat{\xi} = \infty) = 1-m$ contributing $0$ to \eqref{eq: recurrence} for $k \geq 2$. Since $\P(\partial(\cT) \geq j) = F_j(1)$, iterating \eqref{eq: recurrence} from $c_1 = 1$ yields \eqref{eq: describe_c_k}.
	\smallskip
	
	For the closed form \eqref{eq: closed_c_k}, recall $F'(z) = G'(z)$, so by the chain rule
	\begin{align*}
		F_k'(1) = \prod_{j=0}^{k-1} F'(F_j(1)) = F'(1)\prod_{j=1}^{k-1} G'(F_j(1)) = m\,c_k\,,
	\end{align*}
	using $F'(1) = G'(1) = m$. Dividing by $m > 0$ (types~I and~II) gives $c_k = F_k'(1)/m$, which equals $F_k'(1)$ when $m = 1$. The type~III claim is immediate from the infinite-star description.
\end{proof}
\medskip

We now derive an exponential upper bound for $c_k$, which implies that the moments of $\partial(\hat{\cT})$ are finite.

\begin{lem}\label{lem: exponential_c_k}
	Let $(\cT_n)_{n\geq1}$ be a sequence of size-conditioned simply generated trees of type~I or~II. Then there are constants $C, c > 0$, depending only on the offspring distribution, such that
	\begin{align}\label{eq: exponential_c_k}
		c_k = \P(\partial(\hat{\cT}) \geq k) \leq C e^{-ck}\,, \qquad k \geq 1\,.
	\end{align}
	More precisely, $c_k \leq (1 - p_0^2)^{k-1}$ in type~I and $c_k \leq m^{\,k-1}$ in type~II.
\end{lem}
\begin{proof}
	In type~II the function $G'$ is nondecreasing on $[0,1]$, so $G'(F_j(1)) \leq G'(1) = m < 1$, and \eqref{eq: describe_c_k} gives $c_k \leq m^{\,k-1}$.
	\smallskip
	
	In type~I we argue more carefully. First,
	\begin{align}\label{eq: inequality_F}
		F(s) = G(s) - p_0 = \sum_{j=1}^{\infty} p_j s^j \leq s(1 - p_0)\,, \qquad 0 \leq s \leq 1\,,
	\end{align}
	so $F_j(1) \leq (1 - p_0)^{j}$ for every $j \geq 1$. Next, since the offspring distribution is critical,
	\begin{align*}
		G'(s) = \sum_{m \geq 1} m p_m s^{m-1} \leq p_1 + s(1 - p_1)\,, \qquad 0 \leq s \leq 1\,,
	\end{align*}
	and combining both bounds,
	\begin{align*}
		G'(F_j(1)) \leq p_1 + (1 - p_0)^{j}(1 - p_1) \leq p_1 + (1 - p_0)(1 - p_1) \leq (1 - p_0)(1 + p_0)\,,
	\end{align*}
	for $j \in \{1, \dots, k-1\}$, whence $c_k = \prod_{j=1}^{k-1} G'(F_j(1)) \leq (1 - p_0^2)^{k-1}$. Either bound proves \eqref{eq: exponential_c_k}. In the special case $p_1 = 0$ one even has $c_k \leq (1 - p_0)^{(k-1)k/2}$; this is sharper but not always representative. For instance, in the binary case $p_0 = p_2 = 1/2$ one has $c_k = 2^{-2^k + k + 1}$, and the bound $2^{-k^2/2 + k/2}$ is not tight for large $k$.
\end{proof}

The tail $c_k$ has been studied in the literature; see, for example, \cite{Protection1-mean}, \cite{Protection-Clemens} and \cite{Macia2022}, where it is obtained as a limit of the uniform probability over simply generated trees with $\partial \geq k$ and $n$ nodes, via generating functions, power series and asymptotic analysis. The cases $k = 2$ and $k = 3$ for rooted labelled Cayley trees appear in \cite{Ara-Fer}. The proofs given here are elementary and offer a clear probabilistic interpretation of these quantities.
\smallskip

We summarise the moments of $\partial(\hat{\cT})$ in terms of the tail $c_k$.
\begin{propo}\label{propo: moments_kesten_distance_finite}
	Let $\cT_n$ be a sequence of size-conditioned simply generated trees. Then, for every $p>0$,
	\[
	\E\bigl[\partial(\hat{\cT})^{p}\bigr] = \sum_{k=1}^{\infty} \bigl(k^{p} - (k-1)^{p}\bigr)\, c_k\,.
	\]
	In types~I and~II the series is finite by Lemma~\ref{lem: exponential_c_k}; in type~III we have $\partial(\hat{\cT}) = 1$ almost surely and therefore $\E\bigl[\partial(\hat{\cT})^{p}\bigr] = 1$.
\end{propo}
\begin{proof}
	The identity is the standard tail-sum expression for the $\beta$-moment of a nonnegative integer random variable. In types~I and~II the exponential bound of Lemma~\ref{lem: exponential_c_k} makes the series converge; in type~III one has $\partial(\hat{\cT}) = 1$ almost surely, so $c_1 = 1$ and $c_k = 0$ for $k \geq 2$.
\end{proof}

\begin{remark}
	In the critical case ($m = 1$, type~I) the identity given by \eqref{eq: closed_c_k} reads $\P(\partial(\hat{\cT}) \geq k) = F_k'(1)$, and admits a transparent interpretation. Writing $Z_k(\cT)$ for the number of vertices at level $k$ of $\cT$,
	\begin{align*}
		F_k'(1) = \sum_{j \geq 1} j\, \textsc{coeff}_j[F_k(z)]
		&= \sum_{j \geq 1} j\, \P\bigl(Z_k(\cT) = j,\ \partial(\cT) \geq k\bigr)
		\\ &= \E\bigl(Z_k(\cT) \mid \partial(\cT) \geq k\bigr)\, \P(\partial(\cT) \geq k)\,,
	\end{align*}
	and therefore
	\begin{align*}
		\E\bigl(Z_k(\cT) \mid \partial(\cT) \geq k\bigr) = \frac{\P(\partial(\hat{\cT}) \geq k)}{\P(\partial(\cT) \geq k)} = \frac{F_k'(1)}{F_k(1)}\,, \qquad k \geq 0\,,
	\end{align*}
	the logarithmic derivative of $F_k$ at $z = 1$.
	\finremark
\end{remark}

\subsection{Applications}

In this final section, we compute some of these limit constants for several specific cases. In particular, we consider the critical offspring distributions given by Poisson, Geometric, and ${0,2}$-Bernoulli random variables. In the combinatorial setting, these distributions correspond to the uniform rooted labeled Cayley trees, uniform rooted plane trees, and uniform rooted binary plane trees, respectively.
\medskip

Observe that the proof of Proposition \ref{lem: c_k} implicitly establishes the recurrence relation
\begin{equation}
	c_k = G^{\prime}(F_{k-1}(1)) c_{k-1}, \quad \text{for any } k \geq 1,
\end{equation}
with the initial condition $c_0 = 1$.
\medskip

\paragraph{\textbf{Offspring distribution Poisson(1):}} In this case, we have $G(z) = e^{z-1}$, for $|z|\leq1$, which is the probability generating function (PGF) of a Poisson random variable with parameter $\lambda = 1$. Additionally, we have $F(z) = \frac{1}{e} (e^z - 1)$. The recurrence relation in this setting is given by:
\begin{equation}\label{eq: recurrence_cayley}
	c_k = \exp(F_{k-1}(1) - 1) c_{k-1}, \quad \text{ for any } k \geq 1,
\end{equation}
with the initial condition $c_0 = 1$. This recurrence relation also appears in \cite{Macia2022} and \cite{Protection1-mean} in the context of simply generated trees.
\smallskip

We compute the probability $c_k$ for the cases $k = 2, 3, 4$ using Proposition \ref{lem: c_k}. Observe that
\begin{equation}
	c_2 = \P(\partial(\hat{\cT}) \geq 2) = e^{-1/e} \approx 0.6922.
\end{equation}

Using the probability $c_2$ and the recurrence relation \eqref{eq: recurrence_cayley}, we find that
\begin{equation}
	c_3 = \P(\partial(\hat{\cT}) \geq 3) =\frac{1}{e}e^{-1/e}e^{(e^{1-1/e}-1)/e} \approx 0.3522.
\end{equation}
These two constants appear explicitly in \cite{Ara-Fer}[pp. 309-312]. See also \cite{Macia2022} and \cite{Protection1-mean}. To conclude, we compute $c_4$ as follows:

\begin{equation}
	c_4 = \P(\partial(\hat{\cT}) \geq 4) = \frac{1}{e}\exp\left(\frac{1}{e}(e^{(e^{1-1/e}-1)/e}-1)\right)\frac{1}{e}e^{-1/e}e^{(e^{1-1/e}-1)/e}\approx 0.1492.
\end{equation}
This constant also appears in \cite{Macia2022}.
\smallskip

Furthermore, we have approximated some of the moments of the random variable $\partial(\hat{\cT})$ in the $Poisson(1)$ setting using a numerical program based on the recurrence relation \eqref{eq: recurrence_cayley}. We approximate the first three moments as follows:
\begin{align*}
	\E(\partial(\hat{\cT})) \approx 2.28619\,, \quad \E(\partial(\hat{\cT})^2) \approx 6.82517\,, \quad \E(\partial(\hat{\cT})^3) \approx 25.54536\,.
\end{align*}
Therefore, the variance is approximately $\V\mathrm{ar}(\partial(\hat{\cT})) \approx 1.5984$, the previous approximation for the expectation and variance appear in \cite{Protection1-mean}. Using these ideas we can numerically approximate any moment of $\partial(\hat{\cT})$.
\smallskip

\paragraph{\textbf{Offspring distribution Geom(1/2):}} In this case, we have $G(z) = {1}/{(2-z)}$, which is the PGF of a geometric random variable of parameter $p = 1/2$. We also have $F(z) = G(z)-1/2 ={z}/{(4 - 2z)}$, for any $|z|\leq1$. For any $j \geq 0$, we can prove, by induction, that
\begin{equation}
	F_j(z) = \frac{z}{4^j - \frac{8\cdot 4^{j-1}-2}{3}z}\,,
\end{equation}
evaluating $F_j(z)$ at $z = 1$, the expression simplifies to
\begin{equation}
	F_j(1) = \frac{3}{4^j + 2}\,,
\end{equation}
and using the recurrence relation \eqref{eq: recurrence} we find that 
\begin{align*}
	c_k = G^{\prime}(F_{k-1}(1))c_{k-1}\,, \quad \text{ for any } k\geq2\,,
\end{align*}
with $c_0 = 1$. Using the previous recurrence relation we finally obtain that 
\begin{align*}
	c_k = \prod_{j = 1}^{k}\left(\frac{4^{j-1}+2}{2\cdot 4^{j-1}+1}\right)^2\,.
\end{align*}
By working with this product, we obtain the closed form expression
\begin{align*}
	c_k = \P(\partial(\hat{\cT})\geq k) = \left(\frac{3\cdot 2^{k-1}}{2^{2k-1}+1}\right)^2 = \frac{9\cdot 4^k}{(4^k+2)^2}\,, \quad \text{ for any } k\geq 0\,,
\end{align*}
which appears explicitly in \cite{Macia2022}, \cite{Protection-Clemens} and \cite{Protection1-mean}.
\smallskip

The first three (non-trivial) values of $c_k$ are given by 
\begin{align*}
	c_2 = 4/9 \approx 0.44444\,, \quad c_3 = (4/11)^2 \approx 0.13223\,, \quad c_4 = 64/1849 \approx 0.03461\,.
\end{align*}

Using the previous formula for $c_k$ we can estimate the integer moments of $\partial(\hat{\cT})$. For the first, second and third moments of $\partial(\hat{\cT})$ we have:
\begin{align*}
	\E(\partial(\hat{\cT})) &= \sum_{k=1}^{\infty}c_k = 9\sum_{k=1}^{\infty}\frac{4^k}{(4^k+2)^2} \approx 1.62297\,, \\
	\E(\partial(\hat{\cT})^2) &=  \sum_{k=1}^{\infty}(k^2-(k-1)^2)c_k = 9\sum_{k=1}^{\infty}(2k-1)\frac{4^k}{(4^k+2)^2} \approx 3.34973\,,\\
	\E(\partial(\hat{\cT})^3) &= \sum_{k=1}^{\infty}(k^3-(k-1)^3)c_k = 9\sum_{k=1}^{\infty}(3k^2-3k+1)\frac{4^k}{(4^k+2)^2} \approx 8.74174\,,
\end{align*}
and the variance is approximately $\V(\partial(\hat{\cT})) \approx 0.71570$. These quantities appear in \cite{Protection-Clemens} and \cite{Protection1-mean}.
\smallskip

\paragraph{\textbf{Offspring distribution $2\cdot$Ber(1/2):}} In this case, we have that $G(z) = {(1+z^2)}/{2}$, for any $|z|\leq 1$; this is the PGF of a ${0,2}$-Bernoulli random variable with parameter $p = 1/2$. We also have that $F(z) = {z^2}/{2}$, for any $|z|\leq1$. Using the recurrence relation \eqref{eq: recurrence}, we find that
\begin{equation}
	c_k = {c_{k-1}}/{2^{2^{k-1}-1}}, \quad \text{for any } k \geq 1\,,
\end{equation}
where $c_0 = 1$. Iterating this relation, we find that
\begin{equation}
	c_k = \frac{1}{2^{2^k-k-1}}=2^{k+1-2^k}, \quad \text{for any } k \geq 0\,,
\end{equation}
therefore, for any integer $m \geq 1$, we have
\begin{equation}
	\E(\partial(\hat{\cT})^m) = \sum_{k=1}^{\infty} (k^m - (k-1)^m) 2^{k+1-2^k}\,.
\end{equation}

Numerically, we obtain the following approximations for the first three moments:
\begin{equation}
	\E(\partial(\hat{\cT})) \approx 1.56299, \quad	\E(\partial(\hat{\cT})^2) \approx 2.81592, \quad 	\E(\partial(\hat{\cT})^3) \approx 5.70557.
\end{equation}
and the variance $\V\mathrm{ar}(\partial(\hat{\cT})) \approx 0.37298$. Some of these quantities appear in \cite{Protection1-mean}. 

\bigskip

\section{The protection number of a random vertex}\label{sec:protection_random_vertex}

In this section we study the protection number $\partial_{\mathrm{p}}(\cT_n,v_n)$ of a
uniformly chosen vertex $v_n$ of a size-conditioned simply generated tree, that is, the
distance to the border within the fringe subtree rooted at $v_n$, in the three regimes of
Section~\ref{subsec:types}. After fixing the topological setting on the space of pointed
plane trees, we observe that $\partial_{\mathrm{p}}(\cT_n,v_n)$ depends on $\cT_n$ only
through the fringe subtree at $v_n$, which converges in distribution to the
Bienaym\'e--Galton--Watson tree $\cT$; hence
$\partial_{\mathrm{p}}(\cT_n,v_n)\xrightarrow{d}\partial(\cT)$ in all three types,
uniformly and with no extra assumption on the weight sequence. We then prove a universal
exponential tail bound for $\partial_{\mathrm{p}}(\cT_n,v_n)$, adapting the depth-first
encoding of Section~\ref{sec: distance_borde_UI} to the random base point, from which
uniform integrability and the convergence of all moments of
$\partial_{\mathrm{p}}(\cT_n,v_n)$ to those of $\partial(\cT)$ follow.

\subsection{Topology}
Throughout, $v_n$ denotes a uniformly chosen vertex of $\cT_n$ (conditionally on
$\cT_n$), and we work on the compact Polish space $\T_{\infty}^{\bullet}$ of
\emph{pointed plane trees}, equipped with the topology of
\cite[Section~4]{zbMATH07039768}. This space retains the ancestor--descendant
relations and the plane order, and allows the distinguished vertex to have
infinitely many ancestors. The underlying tree spaces $(\T,\delta)$ and
$(\T_{\infty},\delta_{\infty})$ are those introduced earlier.
\smallskip

For a pointed tree $(T,v)$, the \emph{fringe subtree at $v$}, denoted
$\mathrm{fringe}_v(T)$, is the subtree consisting of $v$ and all its descendants,
rooted at $v$. We recall here that the \emph{protection number} of $v$ is the distance to the border of
this fringe subtree,
\[
\partial_{\mathrm{p}}(T,v):=\partial\bigl(\mathrm{fringe}_v(T)\bigr)
=\min_{u\in\mathrm{leaves}(\mathrm{fringe}_v(T))} d(v,u).
\]
Since the leaves of $\mathrm{fringe}_v(T)$ are exactly the leaves of $T$ that are
descendants of $v$, the protection number dominates the distance to the nearest
leaf:
\begin{equation}\label{eq:partial_pV_dominates_partialV}
\partial(T,v)\le\partial_{\mathrm{p}}(T,v),
\qquad\text{where}\qquad
\partial(T,v):=\min_{u\in\mathrm{leaves}(T)} d(v,u).
\end{equation}
The map $(T,v)\mapsto\mathrm{fringe}_v(T)$ is continuous from
$\T_{\infty}^{\bullet}$ to $(\T_{\infty},\delta_{\infty})$ and $\partial$ is Borel
on $\T_{\infty}$; hence $(T,v)\mapsto\partial_{\mathrm{p}}(T,v)$ is Borel measurable.

\subsection{Convergence in distribution}\label{sec:conv_dist_protection}
In this subsection we prove that the protection number $\partial_{\mathrm{p}}(\cT_n,v_n)$
converges in distribution, for weight sequences of types~I, II and~III, to $\partial(\cT)$,
the distance to the border of the Bienaym\'e--Galton--Watson tree $\cT$. The protection
number depends on $\cT_n$ only through the fringe subtree at $v_n$, whose distributional
limit is $\cT$; the convergence is therefore uniform in the three regimes and requires no
assumption on the weight sequence. 

\subsubsection{The limiting objects}
The protection number $\partial_{\mathrm{p}}(\cT_n,v_n)=\partial(\mathrm{fringe}_{v_n}(\cT_n))$
depends on $\cT_n$ only through the fringe subtree at $v_n$. By the fringe-subtree
limit theorem for size-conditioned simply generated trees, the fringe subtree at a
uniformly chosen vertex converges in distribution to the Bienaym\'e--Galton--Watson
tree $\cT$ with offspring distribution $\xi$ (Devroye and
Janson~\cite[Thm.~1.1, Thm.~3.1]{zbMATH06346951}; see also
Janson~\cite[Thm.~7.12]{Janson2012}). Consequently
$\partial_{\mathrm{p}}(\cT_n,v_n)\xrightarrow{d}\partial(\cT)$, and, with the notation
of Section~\ref{sec:local_limits}, $\P(\partial(\cT)\ge k)=F_k(1)$.

\subsubsection{Critical case (type~I)}
\begin{propo}\label{propo:conv_dist_protection_critical}
Let $(\cT_n)_{n\geq1}$ be a sequence of size-conditioned simply generated trees of type~I. Then,
\[
\partial_{\mathrm{p}}(\cT_n,v_n)\xrightarrow{d}\partial(\cT).
\]
Equivalently, for every $k\ge 1$,
\[
\lim_{n\to\infty}\P\bigl(\partial_{\mathrm{p}}(\cT_n,v_n)\ge k\bigr)
=\P\bigl(\partial(\cT)\ge k\bigr)=F_k(1).
\]
\end{propo}
\begin{proof}
The protection number depends on $\cT_n$ only through the fringe subtree at $v_n$, which
converges in distribution to the Bienaym\'e--Galton--Watson tree $\cT$ by the fringe-subtree
limit theorem (Devroye--Janson~\cite[Thm.~1.1, Thm.~3.1]{zbMATH06346951}; see also
Janson~\cite[Thm.~7.12]{Janson2012}). For each fixed $k\ge1$ the event $\{\partial\ge k\}$
depends only on the restriction to level $k$ and is therefore a continuity set for the
limiting law (Lemma~\ref{lem: clopen}); hence
$\P(\partial_{\mathrm{p}}(\cT_n,v_n)\ge k)\to\P(\partial(\cT)\ge k)=F_k(1)$, which is the
asserted convergence in distribution.
\end{proof}

\subsubsection{Condensation case (types~II and~III)}
\begin{propo}\label{propo:conv_dist_protection_subcritical}
Let $(\cT_n)_{n\geq1}$ be a sequence of size-conditioned simply generated trees of type~II or~III. Then,
\[
\partial_{\mathrm{p}}(\cT_n,v_n)\xrightarrow{d}\partial(\cT).
\]
Equivalently, for every $k\ge 1$,
\[
\lim_{n\to\infty}\P\bigl(\partial_{\mathrm{p}}(\cT_n,v_n)\ge k\bigr)
=\P\bigl(\partial(\cT)\ge k\bigr)=F_k(1).
\]
\end{propo}
\begin{proof}
The argument is identical to that of Proposition~\ref{propo:conv_dist_protection_critical}:
the fringe-subtree limit theorem holds for weight sequences of every type, so the fringe
subtree at $v_n$ converges in distribution to $\cT$ (a single vertex in type~III), and the
same continuity-set argument gives
$\P(\partial_{\mathrm{p}}(\cT_n,v_n)\ge k)\to\P(\partial(\cT)\ge k)=F_k(1)$.
\end{proof}

\subsection{Exponential tail bound and uniform integrability}\label{sec:ui_protection}
We now establish a universal exponential tail bound for $\partial_{\mathrm{p}}(\cT_n,v_n)$,
paralleling Proposition~\ref{lem: sub_exp_critical} for the distance from the root.
The proof adapts the strategy of Section~\ref{sec: distance_borde_UI}: we encode the
random choice of $v_n$ as a uniform DFS-index $I$ and use the deterministic observation
that, if $\partial_{\mathrm{p}}(\cT_n,v_n)\ge k$, then the chain of first descendants of
$v_n$ produces $k-1$ non-leaf vertices appearing consecutively in DFS order, starting at
position $I$.

\begin{propo}\label{prop:protection_uniform_vertex_direct}
Let $(\cT_n)_{n\ge 1}$ be a sequence of size-conditioned simply generated trees of
type~I, II or~III. Then there exist constants $n_0\ge 1$ and $C,c>0$, depending only on
the weight sequence, such that
\[
\P\bigl(\partial_{\mathrm{p}}(\cT_n,v_n)\ge k\bigr)\le C\,e^{-ck},\qquad\text{for any }n\ge n_0\text{ and any }k\ge 0.
\]
\end{propo}

\begin{proof}
\emph{type~III.} In this case, consistently with our convention, $p_0=1$.
The exponential concentration estimate for the proportion of leaves used below
remains valid in this regime. Hence type~III is covered by the same argument as
types~I and~II.
\smallskip

The case $k=0$ is trivial (take $C\geq 1$). Fix $k\ge 1$ throughout, and set $m:=k-1$.
\medskip

Run a depth-first traversal of $\cT_n$ and write $D=(d_1,\dots,d_n)$ for its outdegree sequence in DFS, so that $d_i$ is the outdegree of the vertex visited at step $i$. Let $L:=L_n=n_{\cT_n}(0)$ be the number of leaves of $\cT_n$, and let $\mathcal G_n:=\sigma((n_{\cT_n}(j))_{j\ge 0})$ be the $\sigma$-algebra generated by the multiset of outdegrees. By Lemma~\ref{lem: dfs_uniform_Luk}, conditional on $\mathcal G_n$ the sequence $D$ is uniform on the set $\Sigma_M^{\mathrm{Luk}}$ of \L{}ukasiewicz orderings of $M=M(\cT_n)$. We emphasise that $D\mid\mathcal G_n$ is \emph{not} exchangeable as an element of $\Sigma_M$: the partial-sum constraint defining $\Sigma_M^{\mathrm{Luk}}$ breaks the symmetry under arbitrary permutations.
\medskip

Identify the uniformly chosen vertex $v_n$ with the vertex visited at DFS-step $I$, where $I$ is uniform on $\{1,\dots,n\}$ and independent of $\cT_n$ (and hence of $\mathcal G_n$ and $D$).
\medskip

Suppose $\partial_{\mathrm{p}}(\cT_n,v_n)\ge k$. Define recursively the chain of first descendants of $v_n$: set $u_0:=v_n$, and for $i=1,\dots,k-1$, let $u_i$ be the leftmost child of $u_{i-1}$ in $\cT_n$, provided $u_{i-1}$ has at least one child. We claim that this construction produces $k$ vertices $u_0,u_1,\dots,u_{k-1}$, all of them non-leaves except possibly the last.
\medskip

By induction on $i$: each $u_i$ is a descendant of $v_n$ at graph distance $i$ from $v_n$, hence belongs to the fringe subtree at $v_n$. For $i\le k-1$ we have $d(v_n,u_i)=i<k\le\partial_{\mathrm{p}}(\cT_n,v_n)$, so $u_i$ is not a leaf of the fringe subtree at $v_n$. Since the leaves of the fringe subtree at $v_n$ are exactly the leaves of $\cT_n$ that are descendants of $v_n$, $u_i$ is not a leaf of $\cT_n$ either, for $i\le k-2$. In particular $u_{k-1}$ is well-defined as the first child of $u_{k-2}$, and each of $u_0,u_1,\dots,u_{k-2}$ has positive outdegree.
\medskip

By the definition of DFS, the chain $(u_0,u_1,\dots,u_{k-1})$ is visited at consecutive DFS-steps: $u_0$ at step $I$, $u_1$ at step $I+1$, and in general $u_i$ at step $I+i$. In particular $I+m-1\le n$, with all indices read in the ordinary (non-cyclic) sense, and
\begin{equation}\label{eq:vertex_chain_inclusion}
\bigl\{\partial_{\mathrm{p}}(\cT_n,v_n)\ge k\bigr\}\subseteq\bigl\{I+m-1\le n,\ d_I>0,\,d_{I+1}>0,\,\dots,\,d_{I+m-1}>0\bigr\}.
\end{equation}
For $k=1$ the right-hand side reduces to the trivial event $\{I\le n\}$.
\medskip

We bound the probability of the right-hand side of~\eqref{eq:vertex_chain_inclusion} via a random cyclic rotation. Define the rotated sequence
\[
D':=\rho^{I-1}(D)=(d_I,d_{I+1},\dots,d_n,d_1,\dots,d_{I-1}),
\]
where $\rho$ denotes the cyclic left-shift on $\Z_{\ge 0}^n$, so that $D'_j=d_{((I+j-2)\bmod n)+1}$ for $j=1,\dots,n$.

We claim that, conditional on $\mathcal G_n$, the uniformly rotated sequence $D'$ is uniformly distributed on $\Sigma_M$. This is a consequence of the cycle lemma; the uniform distribution of $v_n$ enters in an essential way: only when $I$ is uniform on $\{1,\dots,n\}$ does the cyclic shift $\rho^{I-1}$ produce the uniform measure on $\Sigma_M$. Indeed, the cyclic group $\Z/n\Z$ acts freely on $\Sigma_M$ when $\sum_c c\,M(c)=n-1$: every orbit has size exactly $n$, and by the cycle lemma applied with $r=1$, each orbit contains exactly one \L{}ukasiewicz ordering. Hence $|\Sigma_M^{\mathrm{Luk}}|=|\Sigma_M|/n$, and for every $\sigma\in\Sigma_M$ there is a unique pair $(\tau,j)\in\Sigma_M^{\mathrm{Luk}}\times\{1,\dots,n\}$ with $\sigma=\rho^{j-1}(\tau)$. Since $D$ is uniform on $\Sigma_M^{\mathrm{Luk}}$ and $I$ is uniform on $\{1,\dots,n\}$ independent of $D$, the joint law of $(D,I)$ assigns probability $\frac{1}{n}\cdot\frac{1}{|\Sigma_M^{\mathrm{Luk}}|}=\frac{1}{|\Sigma_M|}$ to each pair, and consequently
\[
\P(D'=\sigma\mid\mathcal G_n)=\frac{1}{|\Sigma_M|}\qquad\text{for every }\sigma\in\Sigma_M.
\]
For a deterministically chosen vertex, such as the root, $I$ is constant and the present argument does not apply; the corresponding tail bound for the root requires the conditional cycle lemma and an identity similar to that in Lemma \ref{lem: master_identity}, as in the proof of Proposition~\ref{lem: sub_exp_critical}.
\medskip

On the event $\{I+m-1\le n\}$, the first $m$ entries of $D'$ coincide with 
\[(d_I,d_{I+1},\dots,d_{I+m-1})\]
in the ordinary (non-cyclic) sense. Therefore
\begin{equation}\label{eq:window_to_rotated}
\P\bigl(I+m-1\le n,\ d_I>0,\dots,d_{I+m-1}>0\bigm|\mathcal G_n\bigr)
\le\P\bigl(D'_1>0,\dots,D'_m>0\bigm|\mathcal G_n\bigr).
\end{equation}
Since $D'$ is uniform on $\Sigma_M$ given $\mathcal G_n$, its first $m$ coordinates form a uniform sample of size $m$ without replacement from $M$. The probability that all $m$ samples are positive is the probability of drawing $m$ items from the $n-L$ non-zero entries of $M$, in order, which by elementary counting equals
\begin{equation}\label{eq:hypergeometric_first_principles}
\P\bigl(D'_1>0,\dots,D'_m>0\bigm|\mathcal G_n\bigr)=\frac{(n-L)_m}{(n)_m},
\end{equation}
where $(x)_m:=x(x-1)\cdots(x-m+1)$ is the falling factorial, with the convention $(n-L)_m/(n)_m=0$ if $m>n-L$.
\medskip

Combining~\eqref{eq:vertex_chain_inclusion},~\eqref{eq:window_to_rotated} and~\eqref{eq:hypergeometric_first_principles},
\begin{equation}\label{eq:vertex_window_bound}
\P\bigl(\partial_{\mathrm{p}}(\cT_n,v_n)\ge k\bigm|\mathcal G_n\bigr)\le\frac{(n-L)_m}{(n)_m}.
\end{equation}

Set $\eta:=L/n$. For each $j\in\{0,1,\dots,m-1\}$,
\[
\frac{n-L-j}{n-j}=1-\frac{L}{n-j}\le \left(1-\frac{L}{n}\right)=1-\eta,
\]
since $n-j\le n$ and $L\ge 0$. Multiplying for $j=0,1,\dots,m-1$ and using $1-x\le e^{-x}$,
\[
\frac{(n-L)_m}{(n)_m}=\prod_{j=0}^{m-1}\frac{n-L-j}{n-j} = \prod_{j=0}^{m-1}\left(1-\frac{L}{n-j}\right)\le(1-\eta)^m\le e^{-m\eta}=e^{-(k-1)\eta}.
\]
Substituting into~\eqref{eq:vertex_window_bound} and taking expectations,
\begin{equation}\label{eq:vertex_R_bound_exp}
\P\bigl(\partial_{\mathrm{p}}(\cT_n,v_n)\ge k\bigr)\le\E\!\left[e^{-(k-1)\eta}\right].
\end{equation}

Fix $c_0\in(0,p_0)$ and split according to whether $\eta\ge c_0$ or $\eta<c_0$:
\[
\E\!\left[e^{-(k-1)\eta}\right]\le e^{-(k-1)c_0}+\P(\eta<c_0).
\]
By Janson's leaf concentration for simply generated trees
(\cite[Theorem~11.4]{Janson2012}, restated as~\cite[Theorem~5.1]{AddarioBerryBrandenbergerHamdanKerriou2022}),
applied with $\varepsilon:=p_0-c_0>0$, there exist $a>0$ and $n_0\ge 1$
such that for every $n\ge n_0$,
\[
\P(\eta<c_0)\le\P(|\eta-p_0|>\varepsilon)\le e^{-an}.
\]

For $n\ge n_0$ and $1\le k\le n$,
\[
\P\bigl(\partial_{\mathrm{p}}(\cT_n,v_n)\ge k\bigr)\le e^{c_0}\,e^{-c_0 k}+e^{-an}\le e^{c_0}\,e^{-c_0 k}+e^{a}\,e^{-ak},
\]
using $e^{-an}\le e^{-ak}$ for $k\le n$. Setting $C:=e^{c_0}+e^{a}$ and $c:=\min(c_0,a)>0$,
\[
\P\bigl(\partial_{\mathrm{p}}(\cT_n,v_n)\ge k\bigr)\le C\,e^{-ck},\qquad 1\le k\le n.
\]
For $k=0$ the bound is trivial (by fixing any $C\ge 1$), and for $k>n$ the probability vanishes since $\partial_{\mathrm{p}}(\cT_n,v_n)\le n-1$ deterministically. The proof is complete.
\end{proof}

\begin{lem}\label{lem:integrabilidad_protection_vertex}
For every $p\ge 0$, $\sup_{n\ge n_0}\E\bigl[\partial_{\mathrm{p}}(\cT_n,v_n)^{p}\bigr]<+\infty$,
and the family $\{\partial_{\mathrm{p}}(\cT_n,v_n)\}_{n\ge 1}$ is uniformly integrable of every
order.
\end{lem}
\begin{proof}
The argument is identical to the proof of Lemma~\ref{lem: integrabilidad}, with
Proposition~\ref{prop:protection_uniform_vertex_direct} in place of
Proposition~\ref{lem: sub_exp_critical}. In particular, for every $p\ge0$ and every $K\ge0$,
\[
\sup_{n\ge n_0}\E\bigl[\partial_{\mathrm{p}}(\cT_n,v_n)^{p}\,\mathbf 1_{\{\partial_{\mathrm{p}}(\cT_n,v_n)>K\}}\bigr]
\le C\sum_{j>K} j^{p}e^{-cj}\ \xrightarrow[K\to\infty]{}\ 0,
\]
so $\sup_{n\ge n_0}\E[\partial_{\mathrm{p}}(\cT_n,v_n)^{p}]<+\infty$ and the family is uniformly
integrable of every order.
\end{proof}

\begin{theo}\label{theo:moments_protection}
Let $(\cT_n)_{n\geq1}$ be a sequence of size-conditioned simply generated trees of type~I, II or~III. Then, for every $p>0$,
\[
\lim_{n\to\infty}\E\bigl[\partial_{\mathrm{p}}(\cT_n,v_n)^{p}\bigr]
=\E\bigl[\partial(\cT)^{p}\bigr]
=\sum_{k\ge 1}\bigl(k^{p}-(k-1)^{p}\bigr)\,F_k(1).
\]
\end{theo}
\begin{proof}
Combine Propositions~\ref{propo:conv_dist_protection_critical} and~\ref{propo:conv_dist_protection_subcritical}
(convergence in distribution) with Lemma~\ref{lem:integrabilidad_protection_vertex}
(uniform integrability of every order); convergence of moments then follows from
Theorem~3.5 of~\cite[p.~30]{MR1700749}.
\end{proof}

\section{The distance to the border of a random vertex}\label{sec:distance_random_vertex}

In this section we study the distance to the border $\partial(\cT_n,v_n)$ of a uniformly
chosen vertex $v_n$ of a size-conditioned simply generated tree, that is, the distance
from $v_n$ to the nearest leaf of the whole tree, in the three regimes of
Section~\ref{subsec:types}. Unlike the protection number, this statistic also sees the
leaves lying outside the fringe subtree at $v_n$, so its local limit is the distance to
the border of the distinguished vertex in the \emph{pointed} local limit $\cT^{*}$ introduced for type I by Aldous~\cite{MR1102319} (see also \cite[Rem. 5.3]{Janson2012})
and for the condensation regimes by Stufler~\cite{zbMATH07039768}; we prove
$\partial(\cT_n,v_n)\xrightarrow{d}\partial(\cT^{*})$ in each type. For the tail we do
not need a separate argument: since $\partial(\cT_n,v_n)\le\partial_{\mathrm{p}}(\cT_n,v_n)$
pointwise, the universal exponential bound for the protection number
(Proposition~\ref{prop:protection_uniform_vertex_direct}) is inherited verbatim by
$\partial(\cT_n,v_n)$, with the same constants. From this inherited bound the uniform
integrability of $\{\partial(\cT_n,v_n)\}_{n\ge1}$ of every order, and hence the
convergence of all moments of $\partial(\cT_n,v_n)$ to those of $\partial(\cT^{*})$,
follow exactly as for the protection number.

\subsection{Limiting trees}

\label{sec:tstar}

The limiting trees for the vicinity of $v_n$ live on a space of pointed plane trees, where trees carry a marked vertex and edges carry a direction so that we can distinguish ancestors from children. The formal space for these objects is constructed in~\cite{zbMATH07039768}. We will not make use of any specific properties of this topology and hence refer the reader to this source for details. The limiting tree $\mathcal{T}^*$ is illustrated in Figure~\ref{fig:pointed_local_limits} and constructed as follows.

For type I, the tree $\cT^*$ consists of a marked vertex $u_0$ which becomes the root of a copy $\cT$. Its parent $u_1$ is created by adding a vertex with  offspring according to an independent copy of $\hat{\xi}$ and identifying a uniformly selected child with $u_0$. The remaining siblings of $u_0$ become roots of independent copies of $\cT$. This construction is repeated to add a parent to $u_1$, and so on,  yielding an infinite number of ancestors of $u_0$.

For type II the construction is analogous, but each ancestor of $u_0$ has the chance to receive an infinite number of children. When this happens for the first time while creating a parent of some vertex $u_K$, $K \ge 0$, we add this infinite number of children to $u_{K+1}$ and identify one of these with $u_K$. All others become roots of independent copies of~$\mathcal{T}$. At this point we stop the construction.

For type III, we let $\cT^*$ denote a vertex with infinitely many children, all of which are leaves, and precisely one of them is declared the marked vertex.

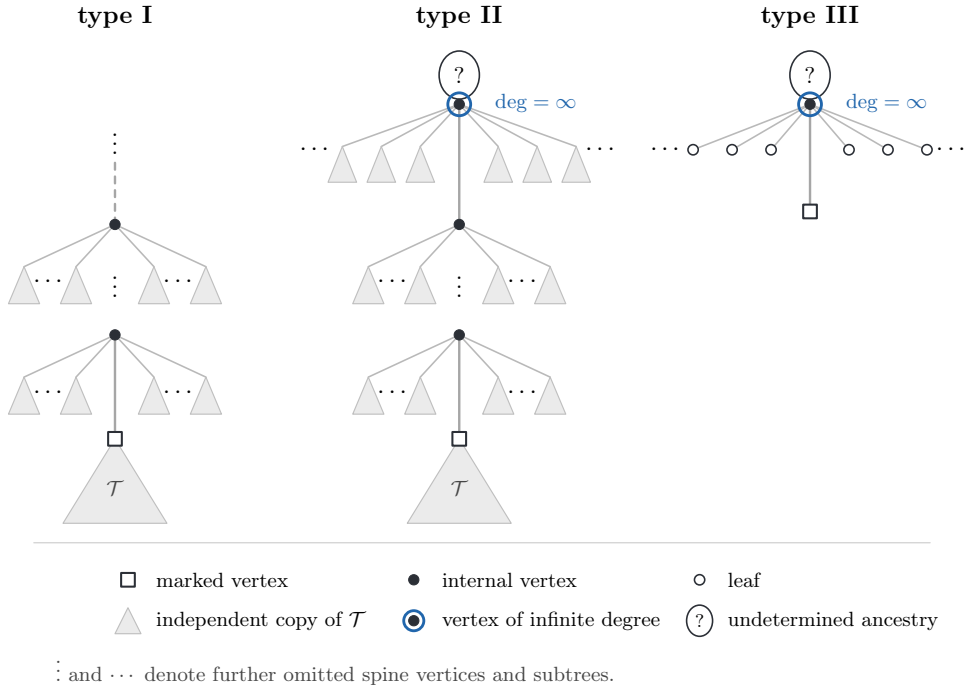
\begin{figure}[H]
\centering
\resizebox{0.9\linewidth}{!}{%
\begin{tikzpicture}[line cap=round,line join=round,
  int/.style ={circle,draw=cInt,fill=cInt,minimum size=4.5pt,inner sep=0},
  leaf/.style={circle,draw=cInt,fill=white,line width=.7pt,minimum size=4.5pt,inner sep=0},
  sq/.style  ={rectangle,draw=cInt,fill=white,line width=.9pt,minimum size=6pt,inner sep=0},
  ed/.style  ={gray!55,line width=0.7pt},
  sp/.style  ={line width=1.05pt,gray!70},
  gw/.style  ={fill=gray!16,draw=gray!50,line width=.5pt},
  big/.style ={pProt,line width=1.25pt}]
\begin{scope}[shift={(0,0)}]
\node[font=\bfseries] at (0,2.05) {type I};
\coordinate (g) at (0,-1.15); \coordinate (u) at (0,-2.85); \coordinate (sq) at (0,-4.45);
\draw[sp,dashed] (g)--(0,-0.15); \node at (0,0.2) {$\vdots$};
\draw[sp] (sq)--(u); \node at (0,-2.0) {$\vdots$};
\foreach \x in {-0.6,-1.4,0.6,1.4}{ \draw[ed] (g)--(\x,-1.8); \fill[gw] (\x,-1.8)--(\x-0.24,-2.37)--(\x+0.24,-2.37)--cycle;}
\node at (-1.0,-2.03) {$\cdots$}; \node at (1.0,-2.03) {$\cdots$};
\foreach \x in {-0.6,-1.4,0.6,1.4}{ \draw[ed] (u)--(\x,-3.5); \fill[gw] (\x,-3.5)--(\x-0.24,-4.07)--(\x+0.24,-4.07)--cycle;}
\node at (-1.0,-3.73) {$\cdots$}; \node at (1.0,-3.73) {$\cdots$};
\fill[gw] (0,-4.45)--(-0.8,-5.75)--(0.8,-5.75)--cycle; \node[font=\footnotesize,gray!55!black] at (0,-5.2){$\cT$};
\node[int] at (g){}; \node[int] at (u){}; \node[sq] at (sq){};
\end{scope}

\begin{scope}[shift={(5.3,0)}]
\node[font=\bfseries] at (0,2.05) {type II};
\coordinate (A) at (0,0.7); \coordinate (g) at (0,-1.15); \coordinate (u) at (0,-2.85); \coordinate (sq) at (0,-4.45);
\draw[sp] (sq)--(u); \draw[sp] (g)--(A); \node at (0,-2.0) {$\vdots$};
\draw[fill=white,draw=cInt,line width=.7pt] (0,1.15) ellipse (0.31 and 0.37);
\node[font=\small] at (0,1.15) {$?$};
\foreach \x in {-0.6,-1.2,-1.8,0.6,1.2,1.8}{ \draw[ed] (A)--(\x,0.05); \fill[gw] (\x,0.05)--(\x-0.22,-0.5)--(\x+0.22,-0.5)--cycle;}
\node at (-2.2,0.0) {$\cdots$}; \node at (2.2,0.0) {$\cdots$};
\draw[big] (A) circle (5pt); \node[int] at (A){};
\node[pProt,font=\footnotesize,anchor=west] at (0.42,0.72) {$\deg=\infty$};
\foreach \x in {-0.6,-1.4,0.6,1.4}{ \draw[ed] (g)--(\x,-1.8); \fill[gw] (\x,-1.8)--(\x-0.24,-2.37)--(\x+0.24,-2.37)--cycle;}
\node at (-1.0,-2.03) {$\cdots$}; \node at (1.0,-2.03) {$\cdots$};
\foreach \x in {-0.6,-1.4,0.6,1.4}{ \draw[ed] (u)--(\x,-3.5); \fill[gw] (\x,-3.5)--(\x-0.24,-4.07)--(\x+0.24,-4.07)--cycle;}
\node at (-1.0,-3.73) {$\cdots$}; \node at (1.0,-3.73) {$\cdots$};
\fill[gw] (0,-4.45)--(-0.8,-5.75)--(0.8,-5.75)--cycle; \node[font=\footnotesize,gray!55!black] at (0,-5.2){$\cT$};
\node[int] at (g){}; \node[int] at (u){}; \node[sq] at (sq){};
\end{scope}

\begin{scope}[shift={(10.7,0)}]
\node[font=\bfseries] at (0,2.05) {type III};
\coordinate (A) at (0,0.7); \coordinate (sq) at (0,-0.95);
\draw[sp] (sq)--(A);
\draw[fill=white,draw=cInt,line width=.7pt] (0,1.15) ellipse (0.31 and 0.37);
\node[font=\small] at (0,1.15) {$?$};
\foreach \x in {-0.6,-1.2,-1.8,0.6,1.2,1.8}{ \draw[ed] (A)--(\x,0.0); \node[leaf] at (\x,0.0){};}
\node at (-2.2,0.0) {$\cdots$}; \node at (2.2,0.0) {$\cdots$};
\draw[big] (A) circle (5pt); \node[int] at (A){};
\node[pProt,font=\footnotesize,anchor=west] at (0.42,0.72) {$\deg=\infty$};
\node[sq] at (sq){};
\node[pVert,font=\footnotesize,anchor=west] at (0.2,-0.95) {};
\end{scope}

\draw[gray!40,line width=.5pt] (-1.25,-6.05)--(12.55,-6.05);
\begin{scope}
  \node[sq] at (0.2,-6.62){};
  \node[anchor=west,font=\footnotesize] at (0.5,-6.62){marked vertex};
  \node[int] at (4.6,-6.62){};
  \node[anchor=west,font=\footnotesize] at (4.9,-6.62){internal vertex};
  \node[leaf] at (9.0,-6.62){};
  \node[anchor=west,font=\footnotesize] at (9.3,-6.62){leaf};
  \fill[gw] (0.2,-7.05)--(-0.0,-7.45)--(0.4,-7.45)--cycle;
  \node[anchor=west,font=\footnotesize] at (0.5,-7.25){independent copy of $\cT$};
  \draw[big] (4.6,-7.25) circle (4.4pt); \node[int] at (4.6,-7.25){};
  \node[anchor=west,font=\footnotesize] at (4.9,-7.25){vertex of infinite degree};
  \draw[fill=white,draw=cInt,line width=.7pt] (9.0,-7.25) ellipse (0.2 and 0.24);
  \node[font=\scriptsize] at (9.0,-7.25){$?$};
  \node[anchor=west,font=\footnotesize] at (9.3,-7.25){undetermined ancestry};
  \node[anchor=west,font=\footnotesize,gray!55!black] at (-1.05,-7.95)
    {$\vdots$ and $\cdots$ denote further omitted spine vertices and subtrees.};
\end{scope}
\end{tikzpicture}%
}
\caption{The pointed local limits}
\label{fig:pointed_local_limits}
\end{figure}

\subsection{Convergence in distribution}\label{sec:conv_dist_vertex}
In this subsection we prove that $\partial(\cT_n,v_n)$ converges in distribution to the distance to the border of the
distinguished vertex in the corresponding pointed local limit.

\subsubsection{Critical case (type~I)}
We prove convergence in distribution of $\partial(\cT_n,v_n)$ to
$\partial(\cT^{*})$ using the local-limit theorem~\cite[Thm.~5.1]{zbMATH07039768}.

\begin{propo}\label{propo:conv_dist_vertex_critical}
Let $\cT_n$ have type I.
Then
\[
\partial(\cT_n,v_n)\xrightarrow{d}\partial(\cT^{*}).
\]
Equivalently, for every $k\ge 1$,
\[
\lim_{n\to\infty}\P\bigl(\partial(\cT_n,v_n)\ge k\bigr)
=\P\bigl(\partial(\cT^{*})\ge k\bigr).
\]
\end{propo}
\begin{proof}
For each integer $k \ge 0$ let $H_k(\cT_n, v_n)$ denote the marked plane tree obtained by forming the fringe subtree at the $k$th ancestor of $v_n$ and marking it at the vertex that corresponds to $v_n$. Thus, $H_k(\cT_n, v_n)$ has a root vertex and a marked vertex. It is possible that $v_n$ has height less than $k$, and in this case we set $H_k(\cT_n, v_n)$ to some placeholder value. The tree $H_k(\cT^*)$ is defined analogously.
\smallskip

It follows from~\cite[Thm.~5.1]{zbMATH07039768} that for each $k \ge 0$
\[
	H_k(\cT_n, v_n) \convd H_k(\cT^*).
\]
From $H_{k+1}(\cT_n, v_n)$ we may determine whether $\partial(\cT_n,v_n)\ge k$ because $H_{k+1}(\cT_n, v_n)$ (among other things) contains all information on the vertices with graph distances at most $k+1$. Likewise, $H_{k+1}(\cT^*)$ determines whether $\partial(\cT^*)\ge k$. It follows that
\[
\lim_{n\to\infty}\P\bigl(\partial(\cT_n,v_n)\ge k\bigr)
=\P\bigl(\partial(\cT^{*})\ge k\bigr).
\]
\end{proof}

\subsubsection{Condensation case (types~II and~III)}
We prove convergence in distribution of $\partial(\cT_n,v_n)$ to
$\partial(\cT^{*})$, treating types~II and~III together
via~\cite[Thm.~5.6]{zbMATH07039768}.

\begin{propo}\label{propo:conv_dist_vertex_subcritical}
Let $\cT_n$  have type II or III. Then
\[
\partial(\cT_n,v_n)\xrightarrow{d}\partial(\cT^{*}).
\]
Equivalently, for every $k\ge 1$,
\[
\lim_{n\to\infty}\P\bigl(\partial(\cT_n,v_n)\ge k\bigr)
=\P\bigl(\partial(\cT^{*})\ge k\bigr).
\]
\end{propo}
\begin{proof}
	In the type III case the marked vertex of $\cT^*$ is almost surely a leaf. Hence $\partial(\cT^{*})=0$. It was shown in~\cite[Thm.~11.4]{Janson2012} that the probability for $v_n$ to be leaf of $\cT_n$ tends to $1$ as $n \to \infty$. Thus,
	\[
		\partial(\cT_n,v_n) \convd 0 = \partial(\cT^{*}).
	\]
	
	Suppose now that we are in the type II case.
	Given a sequence $\Omega_n$ that tends to infinity as $n \to \infty$ and given an integer $k \ge 0$ we let $R_k(\cT_n, v_n)$ denote the marked plane tree obtained as follows. Consider the youngest ancestor $v$ of $v_n$ with degree at least $\Omega_n$. If no such ancestor exists or if $v_n = v$ we set $R_k(\cT_n, v_n)$ to some placeholder value. If such an ancestor $v \ne v_n$ exists we let $u$ denote the unique son of $v$ that lies on the path between $v$ and $v_n$. We set $R_k(\cT_n,v_n)$ to the tree obtained by taking the fringe subtree at $v$ and pruning away all siblings of $u$ (and their fringe subtrees) that lie more than $k$ to the left or right of $u$. We consider $R_k(\cT_n,v_n)$ as marked at the vertex corresponding to $v_n$. 
	
	We define the tree $R_k(\cT^*)$ analogously, with $v^*$ referring to the unique vertex with infinite degree instead, and $u^*$ again the unique son on $v^*$ that lies on the path from $v^*$ to the marked vertex in $\cT^*$. It follows from~\cite[Thm.~5.6]{zbMATH07039768} that for a specific sequence $\Omega_n$ with $\Omega_n \to \infty$ we have for each $k \ge 0$
	\begin{align}
		\label{eq:convo}
	R_k(\cT_n, v_n) \convd R_k(\cT^*).
	\end{align}

	Let $\mathcal{E}_k$ denote the event that one of the children of the root $v$ in $R_k(\cT_n, v_n)$ is a leaf. In this event, $R_k(\cT_n, v_n)$ determines $\partial(\cT_n,v_n)$ and we have $\partial R_k(\cT_n, v_n) = \partial(\cT_n, v_n)$. We let $\mathcal{E}_k^*$ denote the corresponding event for $\cT^*$. Likewise, on $\mathcal{E}_k^*$ it holds that $\partial(\cT^*)$ is determined by $ R_k(\cT^*)$ and  $\partial R_k(\cT^*) = \partial(\cT^*)$.

	In $R_k(\cT^*)$ each child of the root except $u$ is the root of a Bienaym\'e--Galton--Watson tree which may consist of a single vertex with a positive probability $p_0>0$. Thus, the probability for at least one of the children of $v^*$ in $R_k(\cT^*)$ to be a leaf is bounded from below by $1 - (1-p_0)^{2k}$.
	
	Let $\mathcal{E}_k^c$ denote the complementary event to $\mathcal{E}_k$. It follows from~\eqref{eq:convo} that
	\[
		\limsup_{n \to \infty} \P(\mathcal{E}_k^c) \le (1-p_0)^{2k}.
	\]
	Hence by~\eqref{eq:convo} for any set $A$ of nonnegative integers
	\begin{align*}
		\P\bigl(\partial(\cT_n,v_n) \in A\bigr) &= \P\bigl(\partial(\cT_n,v_n) \in A, \mathcal{E}_k\bigr)  + \P\bigl(\partial(\cT_n,v_n) \in A, \mathcal{E}_k^c\bigr) 
	\end{align*}
	with
	\[
		\limsup_{n \to \infty} \P\bigl(\partial(\cT_n,v_n) \in A, \mathcal{E}_k^c\bigr) \le (1-p_0)^{2k}
	\]
	and
	\[
		\P\bigl(\partial(\cT_n,v_n) \in A, \mathcal{E}_k\bigr) \to \P\bigl(\partial(\cT^*) \in A, \mathcal{E}_k^*\bigr)
	\]
	as $n$ tends to infinity. Moreover, 
	\[
		0 \le \P\bigl(\partial(\cT^*) \in A\bigr) -  \P\bigl(\partial(\cT^*) \in A, \mathcal{E}_k^*\bigr) \le  (1-p_0)^{2k}.
	\]
	Since we may take $k$ arbitrarily large, it follows that
	\[
	\P\bigl(\partial(\cT_n,v_n) \in A\bigr) \to \P\bigl(\partial(\cT^*) \in A\bigr)
	\]
	as $n$ tends to infinity. This completes the proof.
\end{proof}

\subsection{Exponential tail bound and uniform integrability}\label{sec:ui_vertex}
The exponential tail bound for $\partial(\cT_n,v_n)$ is inherited from the protection
number through the pointwise domination~\eqref{eq:partial_pV_dominates_partialV}.

\begin{coro}\label{cor:tail_partialV}
 With the constants $n_0,C,c$ of
Proposition~\ref{prop:protection_uniform_vertex_direct},
\[
\P\bigl(\partial(\cT_n,v_n)\ge k\bigr)\le C\,e^{-ck},\qquad\text{for any }n\ge n_0\text{ and any }k\ge 0.
\]
\end{coro}
\begin{proof}
Immediate from $\partial(\cT_n,v_n)\le\partial_{\mathrm{p}}(\cT_n,v_n)$
(see~\eqref{eq:partial_pV_dominates_partialV}) and
Proposition~\ref{prop:protection_uniform_vertex_direct}.
\end{proof}

\begin{lem}\label{lem:integrabilidad_distance_vertex}
For every $p\ge 0$, we have that $\sup_{n\ge n_0}\E\bigl[\partial(\cT_n,v_n)^{p}\bigr]<+\infty$,
and the family $\{\partial(\cT_n,v_n)\}_{n\ge 1}$ is uniformly integrable of every order.
\end{lem}
\begin{proof}
Identical to the proof of Lemma~\ref{lem:integrabilidad_protection_vertex}, using
Corollary~\ref{cor:tail_partialV} in place of
Proposition~\ref{prop:protection_uniform_vertex_direct}.
\end{proof}

\begin{theo}\label{theo:moments_vertex}
Let $\cT_n$ be of type~I, II or~III, and let $\cT^{*}$ denote the corresponding
pointed local limit. Then, for every $p>0$,
\[
\lim_{n\to\infty}\E\bigl[\partial(\cT_n,v_n)^{p}\bigr]
=\E\bigl[\partial(\cT^{*})^{p}\bigr]
=\sum_{k\ge 1}\bigl(k^{p}-(k-1)^{p}\bigr)\,\P\bigl(\partial(\cT^{*})\ge k\bigr).
\]
\end{theo}
\begin{proof}
Combine Propositions~\ref{propo:conv_dist_vertex_critical} and~\ref{propo:conv_dist_vertex_subcritical}
(convergence in distribution) with Lemma~\ref{lem:integrabilidad_distance_vertex}
(uniform integrability of every order); the conclusion follows from Theorem~3.5
of~\cite[p.~30]{MR1700749}.
\end{proof}

\subsection{Calculation of limiting probabilities}\label{subsec: pointed_limit_probs}

Having established that $\partial(\cT_n,v_n)\xrightarrow{d}\partial(\cT^{*})$ in each regime, we determine the law of $\partial(\cT^{*})$ in closed form. The probability is formed by two quantities: the depth tail $F_j(1)=\P(\partial(\cT)\geq j)$ of an ordinary Bienaym\'e--Galton--Watson tree, and the rooted tail $c_k=\P(\partial(\hat{\cT})\geq k)$.

\begin{propo}\label{propo: c_k_star}
	Let $(\cT_n)_{n\geq1}$ be a sequence of size-conditioned simply generated trees and set $c_k^{*}:=\P(\partial(\cT^{*})\geq k)$. Then $c_0^{*}=1$, and in types~I and~II, for every $k\geq1$,
	\begin{align}\label{eq: describe_c_k_star}
		c_k^{*}=F_k(1)\prod_{i=1}^{k-2}G'(F_i(1))=F_k(1)\,c_{k-1}\,,
	\end{align}
	with the conventions that the empty product equals $1$ and $c_0=1$; in particular $c_1^{*}=1-p_0$ and $c_2^{*}=F_2(1)$. This admits the closed form
	\begin{align}\label{eq: closed_c_k_star}
		c_k^{*}=\frac{F_k(1)\,F_{k-1}'(1)}{m}\,,\qquad k\geq2\,,
	\end{align}
	so that $c_k^{*}=F_k(1)\,F_{k-1}'(1)$ for $k\geq2$ when $m=1$ (type~I). In type~III the marked vertex of $\cT^{*}$ is almost surely a leaf, so $\partial(\cT^{*})=0$ almost surely, and hence $c_0^{*}=1$ and $c_k^{*}=0$ for every $k\geq1$.
\end{propo}

\begin{proof}
	We argue in types~I and~II. Recall from \S\ref{sec:tstar} that the marked vertex $u_0$ roots an independent copy $\cT^{(0)}$ of $\cT$, while its ancestors $u_1,u_2,\dots$ are produced one at a time: independently for each $j\geq1$, the parent $u_j$ of $u_{j-1}$ has $\hat{\xi}_j$ children, where $\P(\hat{\xi}_j=r)=r p_r$ for $r\geq1$ and $\P(\hat{\xi}_j=\infty)=1-m$. If $\hat{\xi}_j<\infty$, one of these children, chosen uniformly at random,
	is identified with $u_{j-1}$; if $\hat{\xi}_j=\infty$, the vertex $u_j$ receives infinitely many children on each side of $u_{j-1}$.
	All children other than $u_{j-1}$ root independent copies of $\cT$. In type~I, $m=1$, so $\hat{\xi}_j<\infty$ almost surely and the line of ancestors is infinite; in type~II the construction stops at the first $u_j$ with $\hat{\xi}_j=\infty$, the condensation vertex.
	\smallskip

	The distance from $u_0$ to the nearest leaf splits according to whether that leaf is a descendant of $u_0$ or is reached through an ancestor:
	\begin{align*}
		\partial(\cT^{*})=\min\bigl(D,\,U\bigr)\,,\qquad
		D:=\partial(\cT^{(0)})\,,\qquad
		U:=\min_{j\geq1}\Bigl(j+1+\min_{1\leq i\leq \hat{\xi}_j-1}\partial(\cT^{(j)}_i)\Bigr)\,,
	\end{align*}
	where $\cT^{(j)}_1,\dots,\cT^{(j)}_{\hat{\xi}_j-1}$ are the sibling copies of $\cT$ rooted at the children of $u_j$ other than $u_{j-1}$ (with $\min\emptyset=+\infty$, occurring when $\hat{\xi}_j=1$). In types~I and~II the parent $u_1$ always exists and is never a leaf, since $u_0$ is its child; hence $U\geq2$, and equivalently $U=1+U_{\star}$, the leading $1$ being the obligatory step $u_0\to u_1$ to the always-present parent and $U_{\star}$ the distance from $u_1$ to its nearest leaf not through $u_0$. The variables $D$ and $U$ depend on disjoint and independent parts of the construction, hence are independent, and $\P(D\geq k)=\P(\partial(\cT)\geq k)=F_k(1)$.
	\smallskip

	We compute $\P(U\geq k)$. A leaf reached through $u_j$ lies at distance $j+1+\partial(\cT^{(j)}_i)\geq j+1$; in particular, if $\hat{\xi}_j=\infty$ then almost surely one of the infinitely many sibling copies is a single leaf, contributing a leaf at distance exactly $j+1$. Hence $\{U\geq k\}$ holds if and only if, for every $j\geq1$, every sibling copy at $u_j$ has distance $\geq k-j-1$. For $j\geq k-1$ this is automatic, while for $1\leq j\leq k-2$ it forces $\hat{\xi}_j<\infty$ (an infinite $\hat{\xi}_j$ at such a position would place a leaf at distance $j+1\leq k-1$) together with $\partial(\cT^{(j)}_i)\geq k-j-1$ for all $\hat{\xi}_j-1$ siblings. As these constraints involve independent portions of the construction,
	\begin{align*}
		\P(U\geq k)=\prod_{j=1}^{k-2}\P\bigl(\hat{\xi}_j<\infty,\ \partial(\cT^{(j)}_i)\geq k-j-1 \text{ for all } 1\leq i\leq \hat{\xi}_j-1\bigr)\,.
	\end{align*}
	Conditioning on $\hat{\xi}_j=r$ (probability $r p_r$, $r\geq1$) and using $\P(\partial(\cT)\geq k-j-1)=F_{k-j-1}(1)$, the $j$-th factor equals
	\begin{align*}
		\sum_{r\geq1} r p_r\, F_{k-j-1}(1)^{\,r-1}=G'\bigl(F_{k-j-1}(1)\bigr)\,,
	\end{align*}
	the atom $\P(\hat{\xi}_j=\infty)=1-m$ contributing $0$, consistently with $G'(z)=\sum_{r\geq1}r p_r z^{r-1}$ summing over finite $j$ only. Re-indexing through $i=k-1-j$,
	\begin{align*}
		\P(U\geq k)=\prod_{j=1}^{k-2}G'\bigl(F_{k-1-j}(1)\bigr)=\prod_{i=1}^{k-2}G'(F_i(1))=c_{k-1}\,,
	\end{align*}
	the last identity by \eqref{eq: describe_c_k}. The argument is uniform in the two regimes: in type~II the appearance of the condensation vertex at a position $\geq k-1$ leaves $\{U\geq k\}$ unaffected, while its appearance at a position $\leq k-2$ is precisely excluded by the requirement $\hat{\xi}_j<\infty$ for $1\leq j\leq k-2$. By independence of $D$ and $U$,
	\begin{align*}
		c_k^{*}=\P(D\geq k)\,\P(U\geq k)=F_k(1)\,c_{k-1}\,,
	\end{align*}
	which is \eqref{eq: describe_c_k_star}; for $k=1,2$ it gives $c_1^{*}=F_1(1)=1-p_0$ and $c_2^{*}=F_2(1)$. For the closed form, the chain-rule identity $F_{k-1}'(1)=m\,c_{k-1}$ from the proof of Proposition~\ref{lem: c_k} (valid for $k-1\geq1$) gives $c_{k-1}=F_{k-1}'(1)/m$, whence $c_k^{*}=F_k(1)\,F_{k-1}'(1)/m$ for $k\geq2$; this equals $F_k(1)\,F_{k-1}'(1)$ when $m=1$.
	\smallskip

	Finally, in type~III the limit $\cT^{*}$ is the degenerate star of \S\ref{sec:tstar}, whose marked vertex is one of the infinitely many leaves attached to the condensation vertex; hence $\partial(\cT^{*})=0$ almost surely and $c_k^{*}=0$ for $k\geq1$. This is the limiting form of the fact \cite[Thm.~11.4]{Janson2012} that a uniformly chosen vertex of $\cT_n$ is a leaf with probability tending to $1$.
\end{proof}
\smallskip

In type~II the geometry is the same except that the ascending line of ancestors is finite, terminating at the condensation vertex. Among the infinitely many sibling copies of $\cT$ at that vertex there is almost surely a single leaf, placing a leaf at distance $K+2$ from $u_0$, where $K$ is the number of finite ancestors, and thereby capping the upward route. The finite ancestors $u_1,\dots,u_{k-2}$ nevertheless still contribute the factors $G'(F_i(1))$, so $c_k^{*}=F_k(1)\,c_{k-1}$ is unchanged from type~I.
\smallskip

We now derive an exponential upper bound for $c_k^{*}$, which implies that the moments of $\partial(\cT^{*})$ are finite.

\begin{lem}\label{lem: exp_c_k_star}
	Let $(\cT_n)_{n\geq1}$ be a sequence of size-conditioned simply generated trees of type~I or~II, with $0<p_0<1$. Then
	\begin{align}\label{eq: exponential_c_k_star}
		c_k^{*}=\P(\partial(\cT^{*})\geq k)\leq (1-p_0)^{\,k}\,,\qquad k\geq1\,.
	\end{align}
	In particular, $\partial(\cT^{*})$ has finite moments of all orders.
\end{lem}
\begin{proof}
	By \eqref{eq: inequality_F}, $F(s)\leq s(1-p_0)$ for $0\leq s\leq1$, so $F_k(1)\leq(1-p_0)^{k}$ for every $k\geq1$. Since $c_{k-1}\leq1$, \eqref{eq: describe_c_k_star} gives $c_k^{*}=F_k(1)\,c_{k-1}\leq(1-p_0)^{k}$. As $0<p_0<1$, the right-hand side is summable, and the same comparison applied to $\sum_k k^{p}c_k^{*}$ shows that all moments are finite. The bound may be sharpened to $c_k^{*}\leq(1-p_0)^{k}c_{k-1}$ using Lemma~\ref{lem: exponential_c_k}.
\end{proof}
\medskip

We summarise the moments of $\partial(\cT^{*})$ in terms of the tail $c_k^{*}$.

\begin{propo}\label{propo: moments_pointed}
	Let $(\cT_n)_{n\geq1}$ be of type~I, II or~III. Then, for every $p>0$,
	\[
	\E\bigl[\partial(\cT^{*})^{p}\bigr]=\sum_{k=1}^{\infty}\bigl(k^{p}-(k-1)^{p}\bigr)\,c_k^{*}\,.
	\]
	In types~I and~II the series is finite by Lemma~\ref{lem: exp_c_k_star}; in type~III we have $\partial(\cT^{*})=0$ almost surely and therefore $\E\bigl[\partial(\cT^{*})^{p}\bigr]=0$, for any $p>0$.
\end{propo}

\subsection{Applications}\label{subsec: pointed_applications}

In this final section we compute the pointed constants $c_k^{*}$ for the same three critical offspring distributions considered for the rooted case, Poisson, Geometric and ${0,2}$-Bernoulli. In the combinatorial setting these correspond to a uniformly chosen vertex of a uniform rooted labelled Cayley tree, a uniform rooted plane tree, and a uniform rooted binary plane tree.
\medskip

Observe that the proof of Proposition~\ref{propo: c_k_star} establishes the recurrence relation
\begin{equation}\label{eq: recurrence_star}
	c_k^{*}=F_k(1)\,c_{k-1}\,,\qquad k\geq1\,,
\end{equation}
where $F_k(1)=F(F_{k-1}(1))$ and $c_k=G'(F_{k-1}(1))\,c_{k-1}$ are the recurrences of Proposition~\ref{lem: c_k}, with the initial conditions $F_0(1)=1$ and $c_0=1$; equivalently, $c_k^{*}=F_k(1)F_{k-1}'(1)/m$ for $k\geq2$.
\medskip

\paragraph{\textbf{Offspring distribution Poisson(1):}} In this case $G(z)=e^{z-1}$, $F(z)=\tfrac1e(e^z-1)$ and $m=1$ (type~I). The recurrence \eqref{eq: recurrence_star} reads $F_k(1)=\tfrac1e\bigl(e^{F_{k-1}(1)}-1\bigr)$ and $c_k=\exp\bigl(F_{k-1}(1)-1\bigr)c_{k-1}$, and $c_k^{*}=F_k(1)\,c_{k-1}$. We compute the first values
\begin{equation*}
	c_1^{*}=1-e^{-1}\approx0.63212\,,\quad c_2^{*}\approx0.32432\,,\quad c_3^{*}\approx0.09755\,,\quad c_4^{*}\approx0.01961\,.
\end{equation*}
Numerically, the first three moments are
\begin{equation*}
	\E(\partial(\cT^{*}))\approx1.07727\,,\quad \E(\partial(\cT^{*})^2)\approx2.26436\,,\quad \E(\partial(\cT^{*})^3)\approx5.72421\,,
\end{equation*}
so that $\V\mathrm{ar}(\partial(\cT^{*}))\approx1.10385$.
\smallskip

\paragraph{\textbf{Offspring distribution Geom(1/2):}} In this case $G(z)=1/(2-z)$, $F(z)=z/(4-2z)$, $G'(z)=1/(2-z)^2$ and $m=1$ (type~I). Using $F_j(1)=3/(4^j+2)$ and $c_k=9\cdot4^{k}/(4^{k}+2)^2$, the product \eqref{eq: describe_c_k_star} telescopes to the closed form
\begin{equation*}
	c_k^{*}=\P(\partial(\cT^{*})\geq k)=\frac{27\cdot4^{\,k-1}}{(4^{k}+2)\,(4^{\,k-1}+2)^2}\,,\qquad k\geq1\,.
\end{equation*}
The first values are $c_1^{*}=\tfrac12$, $c_2^{*}=\tfrac16$, $c_3^{*}=\tfrac{2}{99}\approx0.02020$ and $c_4^{*}=\tfrac{8}{5203}\approx0.00154$, and the first three moments are
\begin{equation*}
	\E(\partial(\cT^{*}))\approx0.68851\,,\quad \E(\partial(\cT^{*})^2)\approx1.11276\,,\quad \E(\partial(\cT^{*})^3)\approx2.11421\,,
\end{equation*}
so that $\V\mathrm{ar}(\partial(\cT^{*}))\approx0.63871$.
\smallskip

\paragraph{\textbf{Offspring distribution $2 \cdot$Ber(1/2):}} In this case $G(z)=(1+z^2)/2$, $F(z)=z^2/2$, $G'(z)=z$ and $m=1$ (type~I). Using $F_k(1)=2^{-(2^k-1)}$ and $c_k=2^{\,k+1-2^k}$, \eqref{eq: describe_c_k_star} yields the closed form
\begin{equation*}
	c_k^{*}=\P(\partial(\cT^{*})\geq k)=2^{\,k+1-3\cdot2^{\,k-1}}\,,\qquad k\geq1\,,
\end{equation*}
so that $c_1^{*}=\tfrac12$, $c_2^{*}=\tfrac18$, $c_3^{*}=2^{-8}$ and $c_4^{*}=2^{-19}$. Consequently, for any integer $p\geq1$,
\begin{equation*}
	\E(\partial(\cT^{*})^{p})=\sum_{k=1}^{\infty}\bigl(k^{p}-(k-1)^{p}\bigr)\,2^{\,k+1-3\cdot2^{\,k-1}}\,,
\end{equation*}
and numerically
\begin{equation*}
	\E(\partial(\cT^{*}))\approx0.62891\,,\quad \E(\partial(\cT^{*})^2)\approx0.89454\,,\quad \E(\partial(\cT^{*})^3)\approx1.44929\,,
\end{equation*}
so that $\V\mathrm{ar}(\partial(\cT^{*}))\approx0.49902$.
\smallskip

In each case $\E(\partial(\cT^{*}))<\E(\partial(\hat{\cT}))$ (the rooted means being $\approx2.28619$, $\approx1.62297$ and $\approx1.56299$, respectively), reflecting that a uniformly chosen vertex sits typically closer to the boundary than the size-biased root of the rooted local limit.

\bibliographystyle{abbrv}
\bibliography{main}

\end{document}